\documentclass{amsart}[12pt]

\usepackage{graphicx,url,xspace,enumitem,subfig}
\usepackage[linktocpage]{hyperref}
\usepackage[margin=1in]{geometry}

\newcommand{\R}{\mathbb{R}}

\newcommand{\Z}{\mathbb{Z}}

\theoremstyle{plain}
\newtheorem{theorem}{Theorem}[section]

\newtheorem{lemma}[theorem]{Lemma}
\newtheorem{prop}[theorem]{Proposition}
\newtheorem{cor}[theorem]{Corollary}

\theoremstyle{definition}
\newtheorem{question}[theorem]{Question}
\newtheorem{definition}[theorem]{Definition}

\theoremstyle{remark}

\newtheorem{example}[theorem]{Example}

\numberwithin{equation}{section}

\title{Monodromy factorizations of Seifert fibered spaces}
\author{Victoria Quijano}

\begin{document}
	\begin{abstract}		
		We present relations in the mapping class monoid of $S_{0,0}^n$ between products of boundary parallel twists and those involving only non boundary parallel twists. These are of particular interest because each element gives an open book decomposition of a contact 3-manifold, and different factorizations of the same mapping class give potentially distinct symplectic fillings with diffeomorphic boundaries. We apply these results in order to compute the Euler characteristics of fillings resulting from different factorizations, present plumbing graphs for the fillings given by products of boundary parallel twists, and provide sharp bounds on the Euler characteristics for fillings arising from such relations.
	\end{abstract}
	
	\maketitle
		
	\section{Introduction}
In 2010, Wendl proved that, given a planar open book decomposition of a contact 3-manifold, each minimal symplectic filling is given by a factorization of its monodromy \cite{Wendl}. This is particularly relevant to a process called monodromy substitution, in which one filling is replaced by another in an ambient 4-manifold. Monodromy substitution preserves the symplectic structure on the ambient 4-manifold, but in order to apply it using a relation between factorizations, we must be able to identify the filling corresponding to one of the sides of the relation embedded inside an existing 4-manifold. It is useful to consider relations involving a factorization that corresponds to a symplectic plumbing, because in order to find a plumbing embedded in a 4-manifold, we need only to find a collection of symplectic surfaces that intersect as described by the plumbing graph.

In this article, we consider relations of a specific type in order to identify alternative symplectic fillings of the boundaries of plumbings with star-shaped graphs, called Seifert fibered spaces. In particular, Endo and Gurtas \cite{EndoGurtas} show that monodromy substitution using the lantern relation correlates to a rational blowdown, a process which has been shown to be symplectic by Symington \cite{Symington}. Furthermore, when monodromy subsitution is performed using a plumbing with a star-shaped graph, the process corresponds to star surgery, a generalization of the rational blowdown that was introduced and shown to be symplectic by Karakurt and Starkston \cite{KarakurtStarkston}. We present relations that generalize the lantern relation, which is defined for $\text{Mod}(S_{0,0}^4)$, to the mapping class monoids of planar surfaces with more boundary components. Margalit and McCammond \cite{MargalitMcCammond} describe a well-defined correlation between convex twists on braids and Dehn twists on punctured spheres with exactly one boundary component. In order to prove relations on $\text{Mod}(S_{0,0}^n)$, it suffices to show that the corresponding braids, as defined by \cite{MargalitMcCammond} after the interior boundary components are capped by once-punctured disks, are isotopic and that the multiplicity of each boundary component is the same under both sides of the relation (see \cite{PlamenevskayaStarkston}). We use these results to consider possible relations and discuss the completeness of each list of relations.

In particular, we focus on relations on the mapping class monoid of $S_{0,0}^n$, for $n\geq 5$, of the form $T_{b_1}^{a_1}\cdots T_{b_{n-1}}^{a_{n-1}}T_{b_n}=T_{\alpha_1}\cdots T_{\alpha_k}$ where each $\alpha_i$ is a non boundary parallel curve. We present restrictions on the exponents $a_i$ in order for such a relation to exist. Applying these results, we discuss the graphs of plumbings whose contact boundaries do not have alternative minimal symplectic fillings, concluding that a plumbing given by a star-shaped graph with a central vertex of self-intersection number $-n$ and at most $n-5$ other vertices, all with self-intersection number $-2$ does not have any alternative minimal symplectic fillings. We also prove a generalization of the daisy relation introduced in \cite{GayMark}, describe the plumbing graphs associated with relations of this form, and compute the Euler characteristics of the symplectic fillings associated with each side of the relation in terms of $n$. Additionally, we consider relations of the specified form for particular small values of $n$. A complete set of relations for the $n=5$ case is presented using the generalization of the daisy relation. For $n=6$ and $n=7$, we prove additional relations that are not special cases of this generalization, and we show that the list of relations presented for $n=6$ is complete as well. For each of these sets of relations, we present the associated plumbing graphs and compute the Euler characteristics of the fillings associated with each side of the relation. We also present sharp bounds on the Euler characteristics of plumbings associated with relations of the specified form.

The organization of this article is as follows. We begin with an introduction to the background material and notation that will be used in Section \ref{sec:background}. In Section \ref{sec:relationsthatdonotexist}, we present products of boundary parallel twists that do not have alternative factorizations and the graphs of plumbings that correspond to these products. In Section \ref{sec:relations}, we present the generalization of the daisy relation and the list of relations for $n=5,6,$ and 7 and discuss the  associated Euler characteristics and plumbing graphs. Finally, in  Section \ref{sec:completeness}, we discuss bounds on the Euler characteristics of plumbings associated with relations of the specified form and prove the completeness of the list of relations presented for the $n=5$ and $n=6$ cases.

\vspace{.25cm}

\textbf{\textit{Acknowledgments.}} The author would like to thank Laura Starkston for providing the initial question that led to this project and for her patience, support, and guidance along the way.

\section{Background}
\label{sec:background}
	
Take $S_{g,n}^b$ to be the orientable surface of genus $g$ with $n$ punctures and $b$ boundary components. Recall that the mapping class group of $S_{g,n}^b$ contains the isotopy classes of homeomorphisms from that surface to itself which fix the boundary pointwise and send every puncture in the domain to a puncture in the image. 

In particular, we can use homeomorphisms called Dehn twists to represent elements of the mapping class group of a surface.
	
	\begin{definition}
	Consider the annulus $A=S^1\times [0,1]$ and its embedding into the $(\theta,r)$-plane by the map $(\theta,t)\mapsto(\theta,t+1)$. Induce an orientation onto $A$ using the standard orientation of the plane. Consider the map $T:A\to A$ defined by $T(\theta, t)=(\theta+2\pi  t, t)$. 
	
	Given an oriented surface $S$, and $\alpha$ a simple closed curve in $S$, let $N$ be a regular neighborhood of $\alpha$. Let $\phi:A\to N$ be an orientation-preserving homeomorphism of the annulus.

Then the \textbf{positive Dehn twist about $\alpha$} is the homeomorphism $T_{\alpha}:S\to S$ defined by 

$$T_{\alpha}(x)=\begin{cases}
\phi\circ T\circ \phi^{-1}(x) & {\normalfont
 \text{if }} x\in N \\
 
 x & {\normalfont \text{if }} x\in S/N
\end{cases}$$

A negative Dehn twist about $\alpha$, $T^{-1}_\alpha:S\to S$, is defined by replacing $T$ with $T^{-1}$ in the piecewise function. Note that $T^{-1}(\theta, t)=(\theta-2\pi t, t)$ and that $T_\alpha^{-1}$ is indeed the inverse of $T_\alpha$. 
\end{definition}

	\begin{figure}
		\centering
		\includegraphics[height=2.5cm]{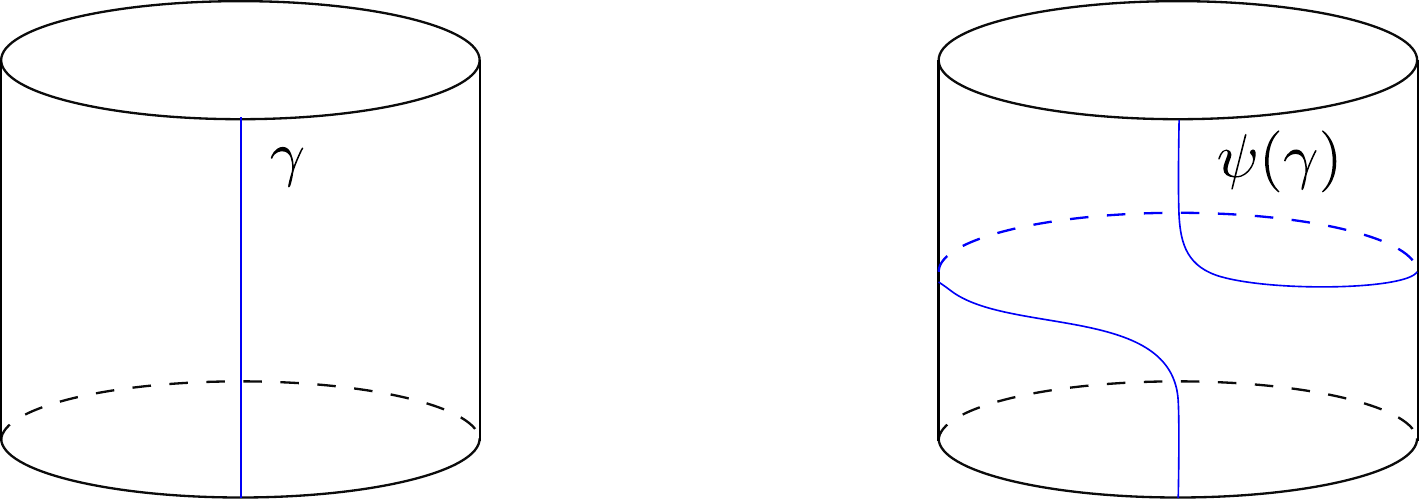}
		\caption{The image on the right shows the image of $\gamma$ after applying the positive Dehn twist $\psi$ to the annulus.} 
		\label{fig:DehnTwist}
\end{figure}

 Dehn twists are well-defined on isotopy classes of curves, so we can use $T_\alpha$ to denote the positive Dehn twist along the isotopy class of $\alpha$. Furthermore, the mapping class group of any surface is generated by Dehn twists over finitely many simple closed curves \cite{FarbMargalit}. Therefore, we can write any element of the mapping class group as a product of Dehn twists.
	
	The mapping class monoid of a surface is the monoid generated by all positive Dehn twists about isotopy classes of essential simple closed curves and boundary parallel simple closed curves on that surface. The mapping class monoids of surfaces are of particular interest because, given a surface and a product of Dehn twists in its mapping class monoid, we can build a 4-dimensional manifold using a Lefschetz fibration.
	
	\begin{definition}
	
	A \textbf{Lefschetz fibration} on a 4-manifold $X$ is a map $\pi: X\to B$, such that $B$ is a compact, connected, oriented surface, $\pi$ has finitely many critical values $t_1,\ldots, t_n$ in the interior of $B$, and in a neighborhood around $t_i$ we can assign complex coordinates $(z_1,z_2)$ such that, in these coordinates, $\pi(z_1,z_2)=z_1^2+z_2^2$.
	
	Furthermore, we can consider the \textbf{fiber}, or preimage, of a point $c\in B$.   For all $c\in B$, if $c$ is not a critical point of $\pi$, then the fiber $\pi^{-1}(c)$ is diffeomorphic to a compact surface. Furthermore, there is a single such surface for all $c\in B$ that are not critical points, and this is called the general fiber \cite{Fuller}. A \textbf{vanishing cycle} in a Lefschetz fibration is the isotopy class of a curve in the general fiber which is collapsed in the fiber of a critical point.
\end{definition}
	
	We will now consider the properties of a Lefschetz fibration.
	 
	\begin{definition}
Given a 3-manifold $M$, an \textbf{open book decomposition of $M$} is a pair $(B,\pi)$ such that $B$ is an oriented link and $\pi: M\backslash B\to S^1$ is a fibration of $M\backslash B$ such that, for all $\theta\in S^1$, $\pi^{-1}(\theta)$ is the interior of a compact surface $S_\theta$ in $M$ and $\partial S_\theta=B$. 

$B$ is called the \textbf{binding} of the open book, and each $S_\theta$ is called a \textbf{page}. 		
	\end{definition}
	
	Conversely, we can begin with an open book and use it to build a 3-manifold.
	
	\begin{definition}
	An \textbf{abstract open book} is a pair $(\Sigma, \phi)$ such that $\Sigma$ is an oriented, compact surface with boundary and $\phi:\Sigma\to\Sigma$ is a diffeomorphism which acts as the identity on a neighborhood of $\partial \Sigma$. $\phi$ is called the \textbf{monodromy} of the open book.
	
	Two abstract open books $(\Sigma_1,\phi_1)$ and $(\Sigma_2, \phi_2)$ are considered equivalent if there exists a diffeomorphism $h:\Sigma_1\to \Sigma_2$ such that $\phi_1\circ h=h\circ \phi_2$.
	\end{definition}
	
	Given a Lefschetz fibration $\pi: X\to B$, we can obtain an open book decomposition of the boundary of $X$ \cite{OzbagciStipsicz}. In this decomposition, the pages are the fibers of the points $x\in \partial B$. In order to obtain the binding, we consider the disjoint union of the boundary components of the fibers. Because each fiber is a surface, its boundary will be diffeomorphic to the union of $n$ copies of $S^1$. The number of boundary components of the general fiber does not change when one of the vanishing cycles is collapsed, so all fibers have the same value for $n$. Therefore, taking the union of the boundaries of each fiber in the Lefschetz fibration will yield $\sqcup_n S^1\times B$. 
	
	The following result describes how to obtain the monodromy of the open book decomposition of the boundary of $X$.
	
	\begin{theorem}[\cite{OzbagciStipsicz}]
	Let $X$ be a 4-manifold, and let $\pi: X\to D^2$ be a Lefschetz fibration on $X$ with $n$ singularities. Then the corresponding abstract open book $(F, \phi)$ has monodromy $\phi=T_{\gamma_1}\cdots T_{\gamma_n}$, where $F$ is the general fiber and each $\gamma_i$ is a vanishing cycle in $F$.
	\end{theorem}
	
	Furthermore, if two abstract open books are equivalent, then their corresponding 3-manifolds are diffeomorphic. In particular, given a surface $F$ with two different factorizations of an element of its mapping class monoid, $\tau$ and $\tau '$, the open books $(F,\tau)$ and $(F,\tau ')$ are equivalent. Let $X$ and $X'$ be the fillings of $(F,\tau)$ and $(F,\tau ')$, respectively. Because the open books are equivalent, $\partial X$ and $\partial X'$ are diffeomorphic. Since $X$ and $X'$ are not necessarily diffeomorphic, this gives us two potentially distinct 4-manifolds with diffeomorphic boundaries. In fact, we will see that the fillings associated with the relations presented in later sections are distinct by computing their Euler characteristics.

A Lefschetz fibration to the disk is called \emph{allowable} if all of its vanishing cycles are essential curves. In this article, all Lefschetz fibrations we consider will be allowable, so we will typically omit this adjective. A correspondence between allowable Lefschetz fibrations and a particularly nice kind of symplectic fillings called Stein fillings was established in~\cite{LP} (see also an alternative proof in~\cite{AO}). A correspondence between contact structures and open book decompositions up to stabilizations was established by~\cite{TW,Giroux}. Given an allowable Lefschetz fibration, the compatible symplectic filling structure induces a contact manifold on the boundary, and under the correspondences, this contact manifold is compatible with the open book induced on the boundary of the Lefschetz
fibration.

We will be restricting our consideration to the case where the general fiber of a Lefschetz fibration has genus 0. In part, this is because it is a simpler starting point in studying mapping class monoids, and it is also partly motivated by the stronger implications for symplectic fillings. We say that an open book decomposition is \emph{planar} if the pages are genus zero. In this case, all (minimal) symplectic fillings, not only Stein fillings, have corresponding Lefschetz fibrations. Moreover, every (minimal) symplectic filling has a Lefschetz fibration such that the open book decomposition induced on its boundary is the fixed planar open book that you started with, as shown in the following theorem of Wendl.
	
	\begin{theorem}[Wendl~\cite{Wendl}] \label{thm:Wendl}
		If $(Y,\xi)$ is a contact $3$-manifold supported by a planar open book $\mathcal{OB}$, then, up to symplectic deformation, every minimal symplectic filling $(W,\omega)$ of $(Y,\xi)$ is supported by a Lefschetz fibration which induces $\mathcal{OB}$ on its boundary.
	\end{theorem}

	As a consequence, if we have a contact manifold supported by a planar open book, then every symplectic filling corresponds to a factorization of the monodromy of that particular planar open book.

	\begin{definition}
	Let $F$ be a closed surface and let $\pi: X\to F$ be a $D^2$-bundle over $F$. Given two disjoint disks in $F$, $D_1$ and $D_2$, we can \textbf{plumb} $X$ at $D_1$ and $D_2$ by identifying $D_1\times D^2$ with $D_2\times D^2$ using a map that interchanges the factors but preserves product structures.
	
	A \textbf{plumbing} is a manifold obtained by finitely many applications of this procedure.
	\end{definition}
	
	Alternatively, we can think of a plumbing as a 4-dimensional thickening of some number of surfaces that are connected by some number of intersections, where any two surfaces must intersect each other at a single point transversally (see \cite{GS}). In the case that all surfaces involved are spheres or disks, we can build a \textbf{plumbing graph} in the following way:
	
	Represent each sphere with a vertex. For any pair of spheres that intersect each other at a point, draw an edge between their corresponding vertices. Finally, label the vertex corresponding to each sphere with the number $-k$, where $k$ is the number of intersections this sphere has with other surfaces. This is called the \textbf{self-intersection number} of the vertex. By convention, we will consider vertices with no label to have a self-intersection number of $-2$.
	
	Furthermore, a Lefschetz fibration with vanishing cycles that are pairwise disjoint builds a plumbing. The collection of surfaces corresponding to this plumbing can be obtained by contracting all vanishing cycles to a point \cite{GayMark}. 
	
	\begin{figure}
		\centering	
	\includegraphics[height=4cm]{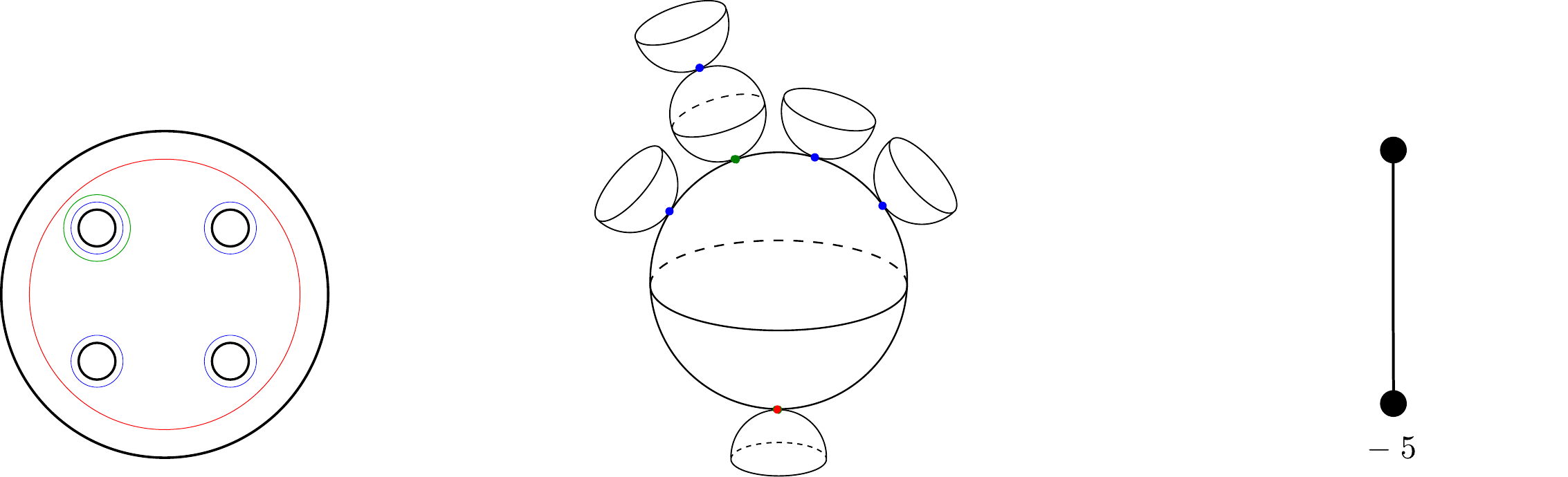}
	\caption{}
	\label{fig:plumbingexample} 
\end{figure}

	\begin{example} Consider the Lefschetz fibration given by taking the surface and the product of twists 
\newline
$T_{b_1}^2T_{b_2}T_{b_3}T_{b_4}T_{b_5}$ on the surface $S_{0,0}^5$. The figure on the left of Figure \ref{fig:plumbingexample} shows the vanishing cycles as colored curves. Collapsing each curve to a point results in the figure in the middle, which completely describes the plumbing that is built by this Lefschetz fibration. On the right, we present the associated plumbing graph.
	\end{example}

	\section{Monodromies with a unique positive factorization}
	\label{sec:relationsthatdonotexist}
	
We will now present monodromy factorizations on surfaces of the form $S_{0,0}^n$ that have no alternative factorizations in the mapping class monoid of the surface. We will focus on mapping class elements of the form $T_{b_1}^{a_1}\cdots T_{b_{n-1}}^{a_{n-1}}T_{b_n}$, using the notation  described below.

Choose an embedding of $S_{0,0}^n$ into $\R^2$ such that the boundary components are arranged convexly and label its boundary components as in Figure \ref{fig:spherenboundary}.

\begin{figure}
\includegraphics[height=3.5cm]{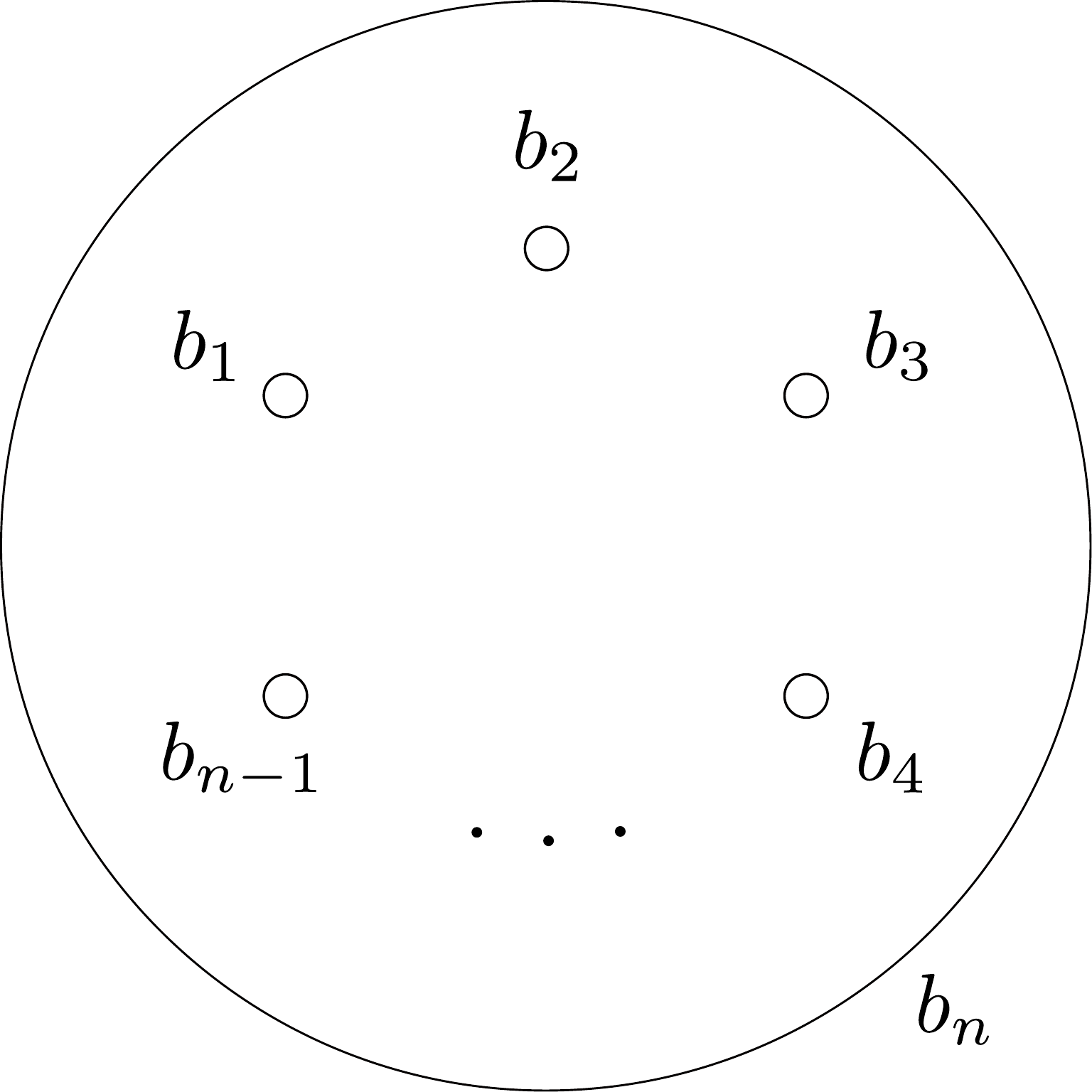}
\caption{}
\label{fig:spherenboundary}
\end{figure}

Let $T_{b_i}$ denote the positive Dehn twist over a simple closed curve that is parallel to the boundary component $b_i$. Additionally, let $T_{b_{i_1},\ldots, b_{i_m}}$ denote the positive Dehn twist over the convex simple closed curve containing the boundary components $b_{i_1},\ldots,b_{i_m}$ when the interior boundary components are arranged convexly. Note that this is well-defined because there is one such curve up to isotopy. When discussing products of twists, we will use function notation; that is, given a product of twists $T_{\alpha_1}\cdots T_{\alpha_m}$, $T_{\alpha_m}$ is applied first, and this pattern continues in descending order of the index on the $\alpha_i$'s.

Furthermore, the following lemma will allow us to use braids to compare products of Dehn twists.

\begin{lemma}
[Margalit and McCammond~\cite{MargalitMcCammond}] \label{thm:MargalitMcCammond}
 
 Let $D_A$ be a convexly punctured disk. In the mapping class group of $D_A$, every convex swing over a set of punctures $B$ is equivalent to the Dehn twist over the convex curve whose interior contains exactly the punctures in $B$. Additionally, every Dehn twist over a convex curve $b$ is equivalent to the convex swing over the points $B$ contained in the interior of $b$.\end{lemma}
 
Using this lemma, we can use braids to represent products of Dehn twists on $S_{0,0}^n$. In particular, we can push them into a semicircle with $b_1$ at the top and $b_{n-1}$ at the bottom, then cap each interior boundary component with a once-punctured disk. This will result in a convexly punctured disk. Then we can use a braid to represent the movement of the boundary components through time as each swing is performed.

Additionally, the \textbf{multiplicity} of a boundary component $b_i$ under a product of Dehn twists $T_{\alpha_1}\cdots T_{\alpha_k}$ is the number of curves $\alpha_j$ that contain $b_i$ in their interiors.

In particular, we will use braids and multiplicity to compare products of Dehn twists using the following lemma.

\begin{lemma}
\label{lemma:linkingnumber}
Two products of Dehn twists over convex curves in $S_{0,0}^n$ are equivalent if and only if their associated braids are isotopic and each interior boundary components have the same multiplicity under each product of Dehn twists.
\end{lemma}

\begin{proof}
If we have two products of Dehn twists on the sphere with boundary that are equivalent, then capping each of the interior boundary components with a punctured disk and performing each product of twists will yield the same braid. Furthermore, the first half of the proof of Lemma 2.3 in \cite{PlamenevskayaStarkston} shows that multiplicity is well-defined under any element of the mapping class group, so the multiplicities of each boundary component under each product must be equal as well. 

The proof of Lemma 2.3 in \cite{PlamenevskayaStarkston} also shows that, if two products of Dehn twists have isotopic braids and each boundary component has the same multiplicity, then the two products of Dehn twists are equivalent in the mapping class group. This completes the argument.
\end{proof}

We will use this result to present new relations and disprove the existence of certain types of relations. In order to prove our first result about relations that do not exist, we will use the following lemma.

\begin{lemma} Given a product of Dehn twists over a collection of curves $\{\gamma_1,\ldots,\gamma_k\}$ on the surface $S_{0,0}^n$, the linking number between any two strands $b_x,b_y$ in the corresponding braid is equal to $-m$, where $m$ is the number of curves $\gamma_i$ that contain both $b_x$ and $b_y$ in their interior.

\end{lemma}

\begin{proof}
Let $\gamma$ be a simple closed curve that contains at least two boundary components. We claim that the braid corresponding to $T_\gamma$ will result in a linking number of $-1$ for each pair $b_i,b_j$ that are both contained in the interior of $\gamma$ and a linking number of $0$ for any other pair of boundary components.

We will first consider the case where $\gamma$ is convex.

Let $\{b_{a_i},\ldots,b_{a_k}\}$ be the set of such boundary components. If we cap all other boundary components with a disk, the braid corresponding to $T_{\gamma}$ will be as shown in Figure \ref{fig:linkingnumber1}. From the first section of this braid, we see that $b_{a_k}$ has two crossings of the value $-1$ with all other strands, and it does not cross any other strands later in the braid. Thus, the linking number between $b_{a_k}$ and any other interior boundary component in $B$ is $\frac{-1-1}{2}=-1$.

\begin{figure}
    \centering
    \begin{minipage}{.5\textwidth}
        \centering
        \includegraphics[height=2.5cm]{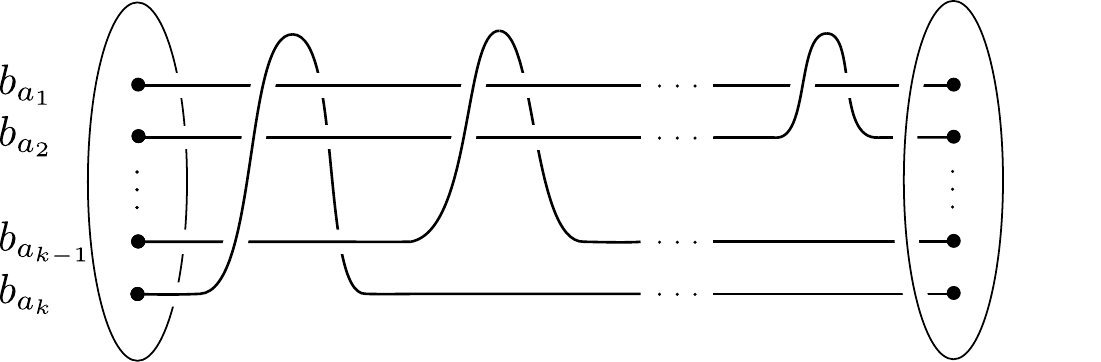} 
\caption{}
\label{fig:linkingnumber1}
        \end{minipage}%
        \begin{minipage}{.4\textwidth}
        \centering
        \includegraphics[height=3.5cm]{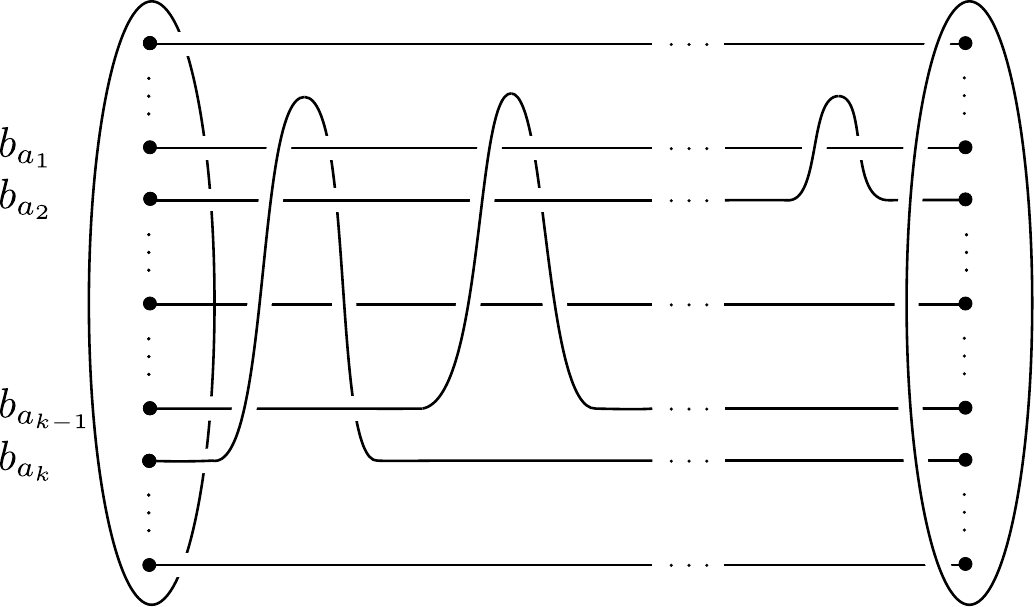} 
\caption{}
\label{fig:linkingnumber2}
    \end{minipage}%
    
\end{figure}

Then $b_{a_{k-1}}$ crosses similarly over all $b_{a_{i}}$ such that $i<k-1$, giving a linking number of $-1$ between $b_{a_{k-1}}$ and any other interior boundary component above it. Furthermore, we already know that $b_{a_k}$ has a linking number of $-1$ with $b_{a_{k-1}}$, so $b_{a_{k-1}}$ has a linking number of $-1$ with every other element of $\gamma$. Continuing inductively, we see that, for any choice of $i$ and $j$, $b_{a_i}$ and $b_{a_j}$ will have a linking number of $-1$.

Now we will consider the linking number between a boundary component in $\gamma$ and another outside of $\gamma$.

Let $b_j$ be the boundary component not in $\gamma$, and let $b_{a_i}$ be given. Then there are three possible positions for $b_j$, as shown in Figure \ref{fig:linkingnumber2}: above all $b_{a_i}$, between the $b_{a_i}$, and below all $b_{a_i}$. If $b_j$ is above or below all boundary components in $b_j$, we see that it has no crossings and therefore has a linking number of 0 with $b_{a_i}$. 

In the case where $a_1<j<a_k$, $b_j$ may be above or below $b_{a_i}$. In the case where it is below, we see that it has no crossings with $b_{a_i}$ and therefore has a linking number of 0 with $b_{a_i}$. In the case where it is above $b_{a_i}$ but below $b_{a_1}$, we see that it has two crossings with $b_{a_i}$: one with a value of $-1$ and the other with a value of $1$. Therefore, the linking number will be $\frac{1-1}{2}=0$.

Furthermore, any two boundary components not in $\gamma$ will have a linking number of $0$ with each other because their strands will stay in place.

Thus, under $T_\gamma$, any two boundary components in the interior of $\gamma$ will have a linking number of $-1$ and any other pair of boundary components will have a linking number of 0.

We will now show that this also holds in the case that $\gamma$ is not convex. Let $\gamma$ be a nonconvex simple closed curve that contains at least two boundary components, labeled $b_{a_1},\ldots b_{a_k}$. We claim that there exists a homeomorphism that takes $\gamma$ to a convex curve $\gamma'$ which contains exactly the same boundary components. First, cut along $\gamma$ to obtain two planar subsurfaces. Then take a homeomorphism to reshape the subsurface containing $b_{a_1},\ldots b_{a_k}$ so that it is convex, while returning each boundary component in its place. We can similarly take a homeomorphism on the other subsurface so that the image of $\gamma$ has the same shape as the outer boundary of the first subsurface and returning each of the original boundary components to its place.

Then glue the subsurfaces back together on the images of $\gamma$ under each homeomorphism. We claim that this operation is a homeomorphism on the whole surface. It is a homeomorphism on each of its subsurfaces, so we need only to ensure that the process of gluing them together is continuous. If it is not continuous along the image of $\gamma$, we can apply an isotopy to a neighborhood of the image of $\gamma$ on one of the subsurfaces to obtain a new function that is continuous. Thus, we have a homeomorphism of the surface that takes a non convex curve $\gamma$ containing $\{b_{a_1},\ldots b_{a_k}\}$ to a convex curve $\gamma'$ containing $\{b_{a_1},\ldots b_{a_k}\}$ and fixes the boundary pointwise.

Therefore, there also exists a homeomorphism $\phi$ that takes $\gamma'$ to $\gamma$. Furthermore, we can choose $\phi$ such that it fixes the boundary pointwise.  Additionally, because $\phi$ is a homeomorphism, $T_{\gamma}=T_{\phi(\gamma')}=\phi T_{\gamma'} \phi^{-1}$ by Fact 3.7 in \cite{FarbMargalit}. Because $\phi$ fixes the boundary pointwise, $\phi$ and $\phi^{-1}$ each belong to mapping classes in $\text{Mod}(S_{0,0}^n)$. We will now consider the braid corresponding to $T_{\gamma}=\phi T_{\gamma'} \phi^{-1}$. Because $\phi$ belongs to a mapping class, it is isotopic to a product of Dehn twists over curves on $S_{0,0}^n$, call it $\tau$. Additionally, $\phi^{-1}$ is isotopic to $\tau^{-1}$. Note that each Dehn twist $T_\alpha$ appears in $\tau$ exactly as many times as $T_{\alpha}^{-1}$ appears in $\tau^{-1}$.

Now we will consider two arbitrary boundary components $b_i$ and $b_j$. By the convex case proven above, we know that the linking number for $b_i$ and $b_j$ from $\tau$ is equal to $-m'$, where $m'$ is the number of curves involved in twists in $\tau$ that contain both $b_i$ and $b_j$. Additionally, because negative Dehn twists correspond to the inverse of a swing on the boundary components in its interior, the linking number between $b_i$ and $b_j$ under $\tau^{-1}$ is $m'$. Therefore, if $b_i$ and $b_j$ are both contained in $\gamma '$, then their linking number under the twist over $\gamma$ is $(-m')+(-1)+m'=-1$. Otherwise, it is $(-m')+0+m'=0$. Therefore, for a non convex curve $\gamma$, under $T_\gamma$, any pair of boundary components both contained in $\gamma$ will have a linking number of $-1$, while all other pairs will have a linking number of $0$.
\end{proof}

We are now equipped to present the following result.
	
	\begin{theorem}\label{thm:exponents1}
		 Let $n\geq 5$, and consider the mapping class monoid of $S^n_{0,0}$. There does not exist a relation between
		$$T_{b_1}\cdots T_{b_n}$$
		and any product of twists over non-boundary parallel curves.
	\end{theorem}

	\begin{proof}
 We proceed by contradiction. Assume that there is a product of Dehn twists over non boundary parallel curves that is equivalent to $T_{b_1}\cdots T_{b_n}$. Then there exists a set of non boundary parallel simple closed curves $B=\{\alpha_1,\ldots,\alpha_k\}$ such that $T_{b_1}\cdots T_{b_n}=T_{\alpha_1}\cdots T_{\alpha_k}$

We know that the braid corresponding to $T_{b_1}T_{b_2}\cdots T_{b_{n-1}}T_{b_n}$ is equivalent to the braid corresponding to $T_{b_n}$. Note that this braid will be as shown in Figure \ref{fig:outerboundary}.

\begin{figure}
\includegraphics[height=2.5cm]{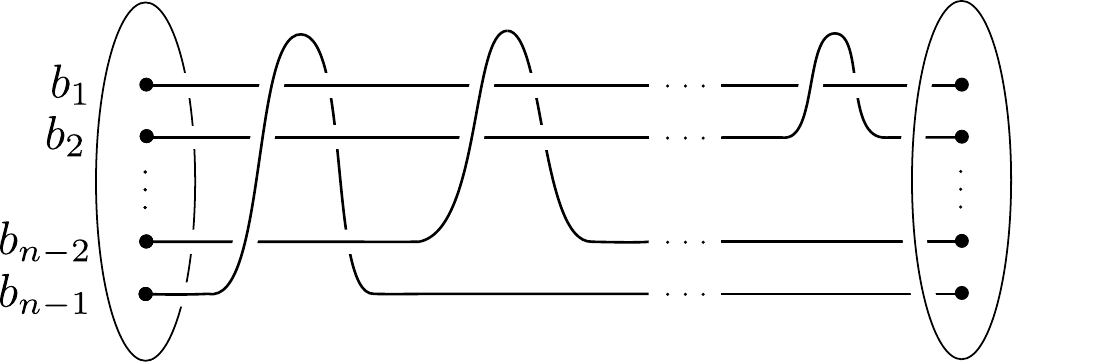} 
\caption{}

\label{fig:outerboundary}
\end{figure}

Additionally, the linking number between any two interior boundary components is $-1$. Since the linking number is a braid invariant, this means that each pair of interior boundary components must have a linking number of $-1$ under the product of twists over curves in $B$. Therefore, by Lemma \ref{lemma:linkingnumber}, for any two interior boundary components $b_i$ and $b_j$, there must be exactly one curve $\alpha_k\in B$ such that both $b_i$ and $b_j$ are contained in its interior.

Furthermore, we see that, under $T_{b_1}T_{b_2}\cdots T_{b_{n-1}}T_{b_n}$, each boundary component has a multiplicity of 2. Therefore, for every interior boundary component $b_i$, there must exist exactly two curves $\alpha_j,\alpha_k\in B$ such that $b_i$ is contained in their interiors.

We claim that, for $n\geq 5$, it is not possible to satisfy both conditions.

We will model the situation using a graph. Build a graph $G$ such that each vertex $v_i$ represents the interior boundary component $b_i$ in $S_{0,0}^{n}$.

We will add edges, and colors of edges, to this graph in the following way: There exists an edge between vertices $v_i$ and $v_j$ if and only if $b_i$ and $b_j$ have a linking number of $-1$. All edges resulting from a particular curve $\alpha_i$ will be the same color, with a distinct color for each $\alpha_i$.

Note that the subgraph of $G$ containing all edges of a particular color, and their adjacent vertices, must be a complete subgraph. This is because, for any two boundary components $b_i,b_j$ contained in a curve $\alpha_k$, the edge between them will be the color corresponding to $\alpha_k$. There is an edge between them because $\alpha_k$ is the only curve containing them both, so they have linking number $-1$ resulting from the twist over $\alpha_k$. Additionally, we note that the number of colors of edges adjacent to a vertex is equal to the multiplicity of the corresponding boundary component.

Then, in order to have a relation with $T_{b_1}T_{b_2}\cdots T_{b_{n-1}}T_{b_n}$, we must have a complete graph. Equivalently, each interior boundary component must have a linking number of $-1$ with every other interior boundary component. Additionally, every vertex must have adjacent edges of exactly two distinct colors (i.e. all interior boundary components have a multiplicity of 2). We will show this is not possible for $n\geq 5$.

Let $n\geq 5$ and consider the complete graph with $n-1$ vertices, $K_{n-1}$. Let $v_1\in K_{n-1}$ be given. We know that $v_1$ must have adjacent edges of exactly two distinct colors, call them blue and green. Then $v_1$ is part of two complete subgraphs, one with blue edges and one with green edges. Since $n-1\geq 4$, at least one of the subgraphs must contain two vertices other than $v_1$. Without loss of generality, assume it is the blue subgraph. Let $v_2,v_3$ be distinct vertices in the blue subgraph and let $v_4$ be a vertex in the green subgraph such that none are equal to $v_1$, as represented in Figure \ref{fig:firstcoloredgraph}.

Because the graph is complete, we know that there must be an edge between $v_4$ and $v_2$, and it must be a third color, call it red. This is shown in Figure \ref{fig:secondcoloredgraph}.

There must also be an edge between $v_4$ and $v_3$, and it cannot be blue or green because the two vertices are not both in either subgraph. Furthermore, it cannot be red because the edge between $v_2$ and $v_3$ is blue, so the red subgraph would not be complete. Thus, it must be a fourth color, call it black. The result of this is shown in Figure \ref{fig:thirdcoloredgraph}.

Now we see that $v_4$ is adjacent to edges of three distinct colors, which contradicts the original conditions. Therefore, such a graph is not possible.
\end{proof}

\begin{figure}
\begin{minipage}{.3\textwidth}
\centering
\includegraphics[height=3cm]{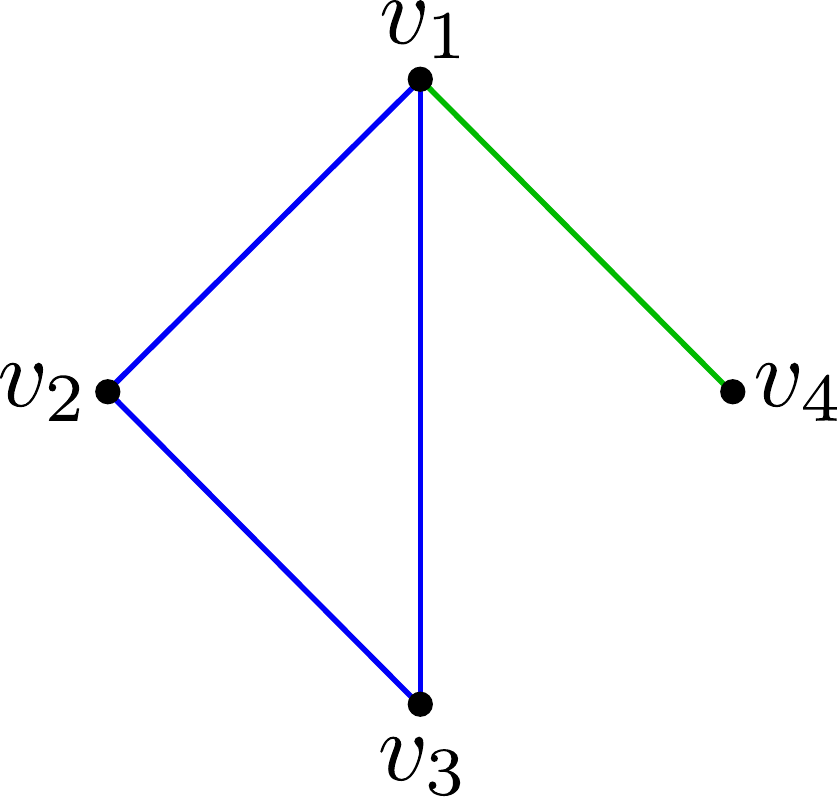} 
\caption{}
\label{fig:firstcoloredgraph}
\end{minipage}
\begin{minipage}{.3\textwidth}
\centering
\includegraphics[height=3cm]{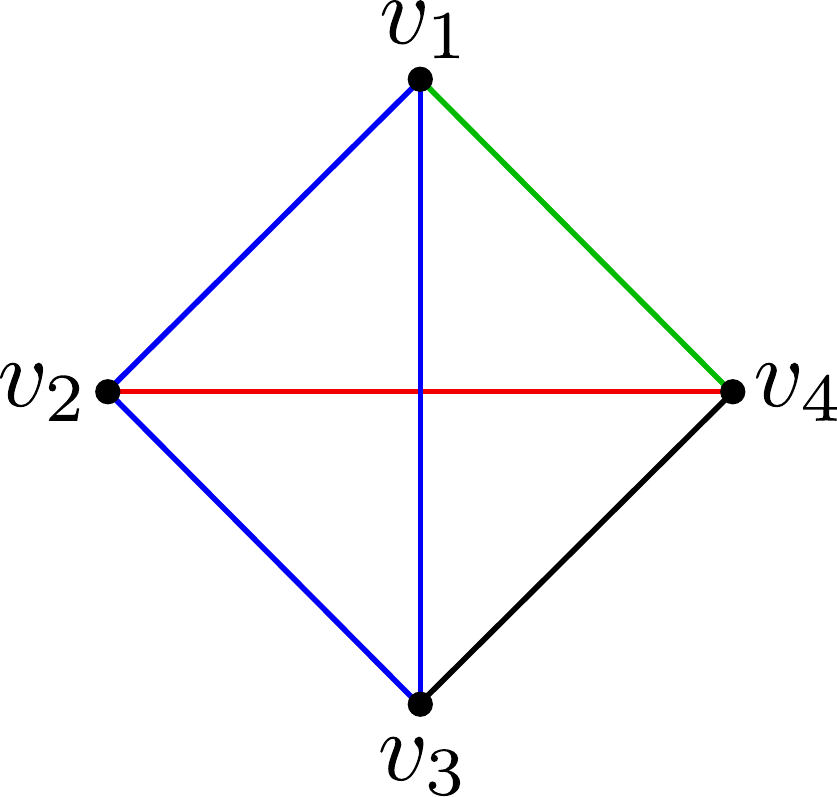} 
\caption{}
\label{fig:secondcoloredgraph}
\end{minipage}
\begin{minipage}{.3\textwidth}
\centering
\includegraphics[height=3cm]{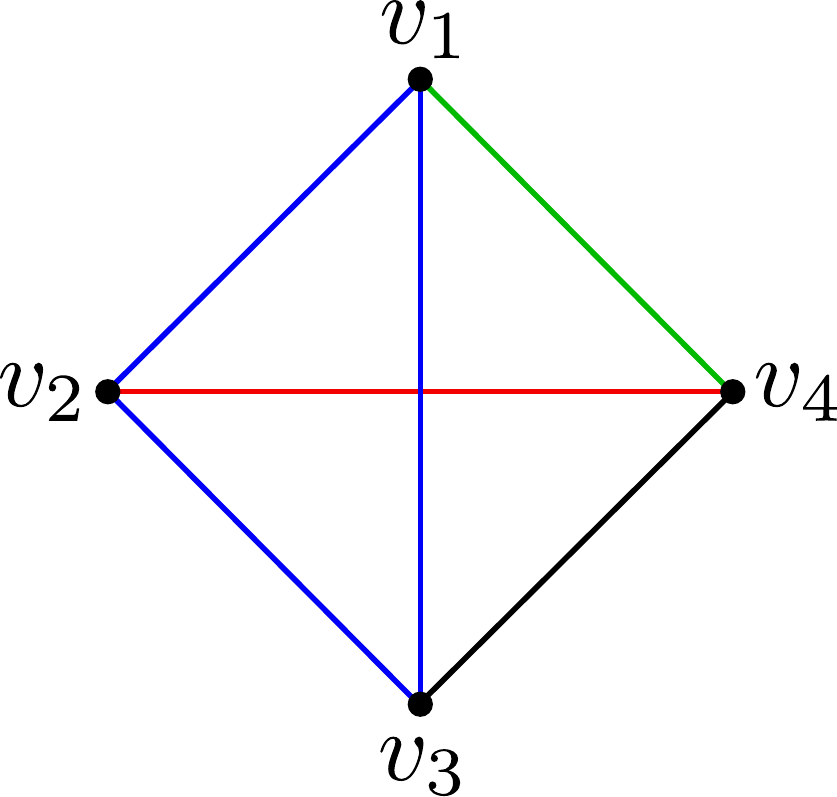} 
\caption{}
\label{fig:thirdcoloredgraph}
\end{minipage}

\end{figure}

	Translating the implications of Theorem~\ref{thm:exponents1} for symplectic fillings, we obtain as a corollary, the following result, which was originally proven by McDuff~\cite[Theorem]{McDuff}.
	\begin{cor}
		For $n\geq 5$, the contact manifold $(Y,\xi)=(L(n,1),\xi_{std})$ has a unique minimal symplectic filling up to symplectic deformation, given by the plumbing whose graph is a single vertex decorated by self-intersection number $-n$. ($L(n,1)$ denotes the lens space, and $\xi_{std}$ is the standard contact structure obtained as a quotient under the $\Z_n$ action from the standard contact structure on $S^3$.)
	\end{cor}

	\begin{proof}
		The open book decomposition with page $\Sigma^n_{0,0}$ and monodromy $T_{b_1}\cdots T_{b_n}$ supports the contact manifold $(L(p,1),\xi_{std})$ by~\cite{PVHM}. By Wendl's Theorem~\ref{thm:Wendl}, any minimal symplectic filling of $(L(p,1),\xi_{std})$ is supported by a Lefschetz fibration which corresponds to a factorization of $T_{b_1}\cdots T_{b_n}$ into positive Dehn twists. By Theorem~\ref{thm:exponents1}, there are no other such positive factorizations, so the only minimal symplectic filling is the one corresponding to the Lefschetz fibration with vanishing cycles $b_1,\dots, b_n$. The construction of Gay and Mark~\cite{GayMark} shows that the symplectic filling with this Lefschetz fibration is the plumbing corresponding to the single vertex with self-intersection number $-n$, which is a single disk bundle over a $2$-sphere with self-intersection number $-n$.
	\end{proof}

Next we give some generalizations of this result. As we will see in the next section, there does exist a relation of the form $T_{b_1}\cdots T_{b_i}^{n-3}\cdots T_{b_n}=T_{\alpha_1}\cdots T_{\alpha_k}$ in the mapping class monoid of $S_{0,0}^n$. However, here we prove that if the exponent $n-3$ is replaced by any smaller number, there is no relation of that form.
	
	\begin{theorem} \label{thm:generalization1}
		Let $n\geq 5$, and consider the mapping class monoid of $\Sigma^n_{0,0}$. Then for any $0<k<n-3$, there does not exist a relation between
		$$T_{b_1}\cdots T_{b_i}^k \cdots T_{b_n}$$
		and any product of twists over non-boundary parallel curves.
	\end{theorem}

	\begin{proof}
		
Assume for the sake of contradiction that there does exist such a relation. Then we know there exist non boundary parallel curves $\gamma_i$ such that 

$$T_{b_1}\cdots T_{b_i}^{k}\cdots T_{b_n}=T_{\gamma_1}\cdots T_{\gamma_l}$$

As in the proof of Theorem 3.2, we will represent the product of twists on the right hand side using a graph $G$. Let there be one vertex in $G$ for each interior boundary component, and let there be an edge between any two vertices if they are contained within one of the curves $\gamma_i$. Designate a distinct color for each $\gamma_i$, and color the edges of the graph so that the colors correspond to the curve they represent.

In particular, this means that there are $n-1$ vertices in $G$. We know that $n\geq k+4$, so $n-1\geq k+3$.

Thus, there are at least $k+3$ vertices in $G$.

Additionally, because there must be exactly one curve containing any pair of boundary components in the collection $\gamma_1,\cdots,\gamma_l$, we know that $G$ is a complete graph. 

Now we consider $b_i$. Because $G$ is complete and has at least $k+3$ vertices, $b_i$ is incident to at least $k+2$ edges.

Additionally, we note that the multiplicity of $b_i$ under the product of twists on the left hand side is $ k+1$. In particular, this means that the set of edges incident to $b_i$ must contain edges of exactly $k+1$ distinct colors. Thus, by the pigeonhole principle, there exist at least two edges adjacent to $b_i$ that are of the same color, call it blue.

This means that there is a complete subgraph with blue edges and at least three vertices. Call the other vertices $b_x$ and $b_y$.

Let $b_j$ be a vertex not in this subgraph. We know such a vertex exists because $b_i$ is incident to at least one edge that is a color other than blue.

Because $G$ is complete, there must be edges between $b_j$ and each vertex of the blue subgraph, and these edges cannot be blue. Because a subgraph of any color must be complete, and these edges are not blue, each one must be of a distinct color from the others. 

Thus, $b_j$ has a multiplicity of at least 3, which produces a contradiction.
	\end{proof}
	
	Combining Theorem~\ref{thm:generalization1} with Wendl's Theorem~\ref{thm:Wendl}, we immediately obtain the following corollary.
	
	\begin{cor}
		Let $(Y,\xi)$ be the contact boundary of the symplectic plumbing corresponding to the graph shown below, where $k< n-3$.
		\begin{center}
		\includegraphics[height=1cm]{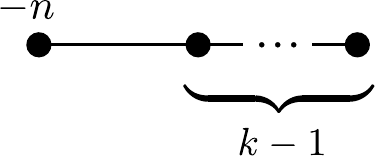} 
		\end{center}
		Then this symplectic plumbing is the unique minimal symplectic filling of $(Y,\xi)$ up to symplectic deformation.
	\end{cor}

We will now present a further generalization of this result.  We proved that products of boundary-parallel Dehn twists on a surface with a particular collection of exponents on $T_{b_i}$ would not yield a relation with twists over non boundary parallel curves. We will now prove that, given a product of boundary parallel Dehn twists, if the sum of the exponents is less than or equal to the sum of the particular collection of exponents presented in Theorem 3.7, then there is also no relation with a product of non boundary parallel twists.

\begin{theorem} \label{thm:generalization2} Let $n\geq 5$ and consider the mapping class monoid of $S_{0,0}^n$. Given a product of boundary parallel twists of the form

$$T_{b_1}^{a_1}\cdots T_{b_{n-1}}^{a_{n-1}}T_{b_n}$$

where each $a_i$ is a positive integer, if $1+\sum_{i=1}^{n-1} a_i \leq 2n-5$, then there is no relation between the given product and any product of twists over non boundary parallel curves.
\end{theorem}

\begin{proof} We will again model the situation using a graph, as described in the proof of Theorem \ref{thm:exponents1}. We claim that, for the complete graph with $n-1$ vertices $b_i$, there is no way to partition the edges by color into strict complete subgraphs such that $1+\sum a_i\leq 2n-5$, where $a_i$ is the number of distinct colors represented by the edges incident to the vertex $b_i$ minus 1.

We proceed by induction on $n$, with the base case $n=5$. In this case, we will consider all possible ways to partition the edges of the complete graph with 4 vertices, $K_4$ into complete subgraphs. 

The largest possible strict complete subgraph has three vertices. In the case that three of the edges are part of a complete subgraph with three vertices, up to relabeling and using arbitrary colors, we have the graph shown in Figure \ref{fig:standalonegraph}.

Then $1+\sum a_i=1+(2-1)+(3-1)+(2-1)+(2-1)+(2-1)=1+1+2+1+1+1=6>2(5)-5$

The only other possible case is to color each edge differently. Then each $a_i$ would be $4-1=3$, so $1+\sum a_i=1+4(3)=13>2(5)-5$.

Thus, the claim holds for $n=5$.

We now proceed with the inductive step. Assume that, for some $k=n$ such that $k\geq 5$, the hypothesis holds. In particular, for $n=k$, there does not exist a partition of the edges of $K_{n-1}$ into complete subgraphs such that $1+\sum a_i\leq 2n-5$, with $a_i$ as defined above. We will show that this is true for $n=k+1$.

Assume for the sake of contradiction that there does exist such a partition by color on $K_k$ into complete strict subgraphs such that $1+\sum a_i\leq 2(k+1)-5=2k-3$.

We claim that, given such a partition on $K_k$, there must exist at least one vertex $b_j$ such that the edges incident to $b_j$ represent at least three distinct colors.

In the case that all subgraphs have exactly two vertices, all edges incident to any particular vertex must be of distinct colors. Since the graph is complete, all vertices are incident to exactly $k-1$ vertices. Therefore, the edges incident to any $b_i$ represent $k-1$ vertices. Since $k\geq 5$, this means that all vertices are incident to edges of at least four distinct colors. 

In the case that there is at least one subgraph with three vertices, $v_1,v_2,$ and $v_3$, we consider a vertex outside of that subgraph, $v_4$. Because $v_4$ is not part of the subgraph, the edge between $v_1$ and $v_4$ cannot be the same color as the subgraph containing $v_1,v_2,$ and $v_3$. Similarly, the edge between $v_2$ and $v_4$ and that between  $v_3$ and $v_4$ must each have their own distinct color, so $v_4$ has is incident to edges that represent at least three distinct colors.

Thus, in our partition of $K_k$, we can choose a vertex $b_j$ that is incident to edges of at least three distinct colors. In this case, $a_j\geq 3-1=2$.

Consider the graph resulting from removing $b_j$ and the edges incident to it. Because removing a vertex from a complete graph $K_m$ gives the complete graph $K_{m-1}$, we know that the resulting graph, and the subgraphs in the partition, are complete.

Then we have a complete graph with $k-1$ vertices $b_1',\ldots, b_{k-1}'$. 

In the case that the associated coloring of the edges of this new graph are all one color, we know that all other subgraphs of $K_k$ must be edges from $b_j$ to vertices in this subgraph with $k-1$ vertices. This means that $a_j=k-2$, and $a_i=1$ for all $i\neq j$ so $1+\sum a_i=1+k-2+(k-1)(1)=2k-2> 2k-3$. Thus, this is not possible. Therefore, the associated partition of the resulting graph will still contain strict complete subgraphs.

Additionally, $a_j$ is not part of the sum $1+\sum a_i'$, and, given a vertex $b_i$ and its associated vertex $b_m'$ in $K_{k-1}$, $a_i\geq a_m'$.

Thus, $1+\sum a_i'\leq (1+\sum a_i)-a_j\leq 2k-3-2=2k-5$ 

This contradicts the inductive hypothesis, so the hypothesis must also be true for $n=k+1$.

Therefore, for the complete graph with $n-1$ vertices $b_i$, there is no way to partition the edges by color into strict complete subgraphs such that $1+\sum a_i\leq 2n-5$.

This implies that, given $a_i$ such that $1+\sum a_i\leq 2n-5$ and any combination of non boundary parallel curves $\gamma_i$ such that $T_{\gamma_1,\cdots,\gamma_m}$ corresponds to a braid with a pairwise linking number of $-1$ for all interior boundary components, it is not possible for the multiplicity of all interior boundary components to be equal under both $T_{\gamma_1}\cdots T_{\gamma_m}$ and $T_{b_1}^{a_1}\cdots T_{b_{n-1}}^{a_{n-1}}T_{b_n}$.

Therefore, given $a_i$ satisfying $1+\sum a_i\leq 2n-5$, there is no relation between 

$$T_{b_1}^{a_1}\cdots T_{b_{n-1}}^{a_{n-1}}T_{b_n}$$

and any product of twists over non boundary parallel curves.
\end{proof}

We now consider the plumbing graph associated with a product of boundary parallel twists as described in Theorem \ref{thm:generalization2}. By the construction of plumbing graphs, the graph for any such factorization is star shaped and the central vertex has self-intersection number $-n$. Each factor $T_{b_i}^{a_i}$, for $i\leq n-1$, corresponds to a linear subgraph with $a_i-1$ vertices, each with self intersection number $-2$, where exactly one of these vertices is adjacent to the central vertex. Additionally, all such subgraphs are disjoint. This means that there are $\sum_{i=1}^{n-1}(a_i-1)$ vertices with self-intersection number $-2$. Since $1+\sum a_i\leq 2n-5$, we can conclude that $\sum_{i=1}^{n-1}(a_i-1)\leq n-5$, so there must be at most $n-5$ vertices other than the central vertex. , note that, for $n\geq 5$, every star shaped graph with a central vertex of self-intersection number $-n$ and at most $n-5$ other vertices, each with self intersection number $-2$, describes a plumbing given by a monodromy factorization of the form described in Theorem \ref{thm:generalization2}. Combining this with Wendl's Theorem \ref{thm:Wendl}, we obtain the following corollary.

\begin{cor}
Let $(Y,\xi)$ be the contact boundary of the symplectic plumbing corresponding to a star-shaped graph with at most $n-4$ vertices, where the central vertex has self-intersection number $-n$ and all others have self-intersection number $-2$. Then this symplectic plumbing is the unique minimal symplectic filling of $(Y, \xi)$ up to symplectic deformation.
\end{cor}
	
	\section{Relations in the planar monoid}
	\label{sec:relations}
	
	In this section we consider situations in which there exist nontrivial relations between different factorizations of positive Dehn twists. In particular, in order to admit a plumbing, a product of Dehn twists must involve only disjoint curves, and we are interested in relations between factorizations that admit plumbings and those that do not. For the scope of this project, we will focus on relations between products of the form $T_{b_1}^{a_1}\cdots T_{b_{n-1}}^{a_{n-1}}T_{b_n}$, where each $a_i$ is a positive integer, and products of twists over convex, non boundary parallel curves.
	
	\begin{figure}

\includegraphics[height=3cm]{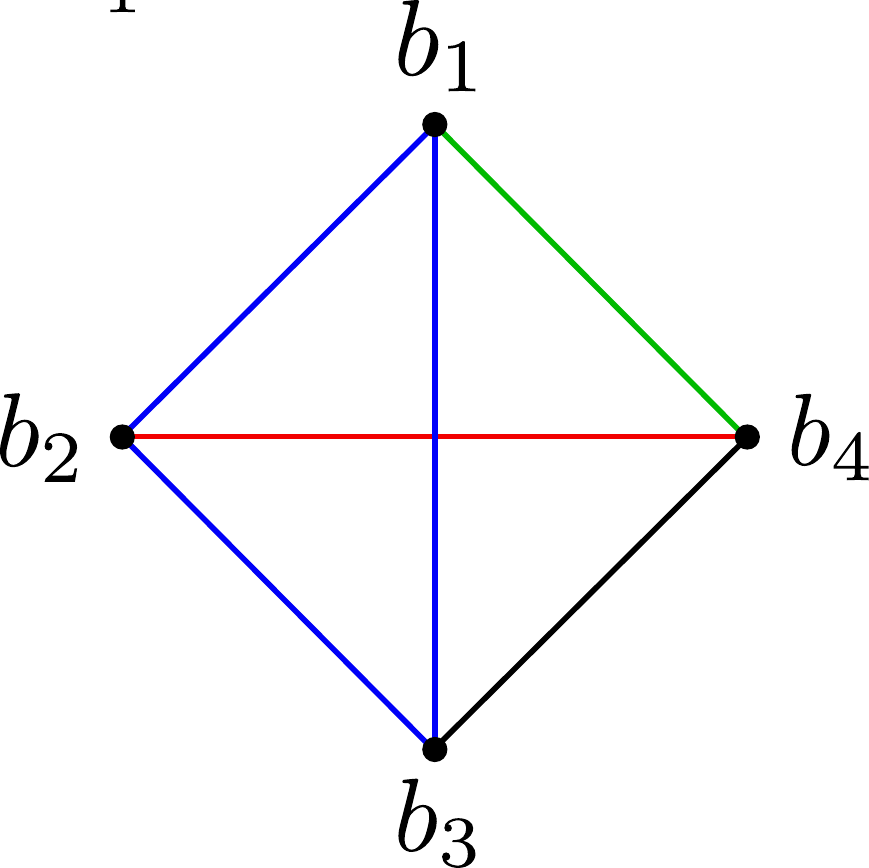} 
\caption{}
\label{fig:standalonegraph}
\end{figure}

	\subsection{Relations for general \texorpdfstring{$S^n_{0,0}$}{Sn0,0}}
	
	We start by giving relations that hold in the planar mapping class group for an arbitrary number of boundary components. 
	
	\begin{theorem} \label{thm:relationgeneral}
		Consider the mapping class monoid of $S_{0,0}^n$. For all $2\leq i<n-1$ the following relation holds:
$$[T_{b_1}^{n-i-1}T_{b_2}^{n-i-1}\cdots T_{b_i}^{n-i-1}][T_{b_{i+1}}^{n-3}\cdots T_{b_{n-1}}^{n-3}]T_{b_n}=T_{b_1,b_2,\ldots,b_i}[T_{b_{i+1},b_{i}}\cdots T_{b_{i+1},b_1}]\cdots[T_{b_{n-1},b_{n-2}}\cdots T_{b_{n-1},b_1}]$$
	\end{theorem}

	Note that this relation generalizes the \emph{daisy relation} defined in~\cite{EMVHM}.

\begin{proof}
We will first show that the multiplicities are equal. Let $b_j$ be an interior boundary component. 

If $j\leq i$, then on the left hand side, the twists in the product involving $b_j$ are $T_{b_j}^{n-i-1}$ and $T_{b_n}$. All other components of the product involve twists over curves that do not contain $b_j$. Thus, the multiplicity of $b_j$ is $n-i-1+1=n-i$ on the left hand side. On the right hand side, the twists involving $b_j$ are $T_{b_1,b_2,\ldots,b_i}$ and the $T_{b_j,b_k}$ contained in $[T_{b_{k-1},b_k}\cdots T_{b_1,b_k}]$ for each $k>i$. We know that there are $n-i-1$ values of $k$ that are greater than $i$, so the multiplicity of $b_j$ is $n-i-1+1=n-i$.

If $j>i$, then on the left hand side, then the twists in the product involving $b_j$ are $T_{b_j}^{n-3}$ and $T_{b_n}$. Therefore, the multiplicity of $b_j$ on this side is $n-3+1=n-2$. On the right hand side, twists over all curves containing $b_j$ and exactly one other interior boundary component are included. Since there are exactly $n-2$ such curves, the multiplicity of $b_j$ is $n-2$.

Thus, the multiplicities are equal.

Now we will show that the corresponding braids are isotopic.

We will begin by finding the braid corresponding to the left hand side. We note that all twists except for $T_{b_n}$ are boundary parallel. Because a twist along a boundary parallel curve corresponds to a braid that demonstrates a swing on the single puncture contained within that curve, it will not add any crossings to the braid. Therefore, the braid corresponding to the left hand side is exactly the braid corresponding to just $T_{b_n}$. 

$T_{b_n}$ is represented by a swing on all interior boundary components, as shown in Figure \ref{fig:outerboundarytwistbraid}.

We will now find the braid corresponding to the right hand side.

We note that each portion in brackets is of the form $[T_{b_j,b_{j-1}}\cdots T_{b_j,b_1}]$ and goes from $j=n-1$ to $j=i+1$.

We will consider the braid corresponding to $T_{b_j,b_{k+1}}T_{b_j,b_{k}}$ for any $k<j$. Performing the required swings gives us the braid in Figure \ref{fig:firstbraid}.

Cancelling out crossings that are inverses of each other gives an isotopy that leads to the braid in Figure \ref{fig:midisotopy}. This is isotopic to the braid in Figure \ref{fig:endisotopy}.

We will apply the same isotopy inductively to the braid corresponding to $[T_{b_j,b_{j-1}}\cdots T_{b_j,b_1}]$ for arbitrary $j$. We know that $[T_{b_j,b_{j-1}}\cdots T_{b_j,b_1}]$ is represented by the braid in Figure \ref{fig:uncancelledtwists}, and applying the isotopy yields the braid in Figure \ref{fig:cancelledtwists}. 

Thus, $[T_{b_{i+1},b_{i}}\cdots T_{b_{i+1},b_1}]\cdots[T_{b_{n-1},b_{n-2}}\cdots T_{b_{n-1},b_1}]$ is represented by the braid shown in Figure \ref{fig:firstnminusitwists}.

Furthermore, $T_{b_1,\cdots,b_i}$ corresponds to the braid in Figure \ref{fig:lastitwists}.

Putting these together in the order they are shown on the right hand side yields the braid in Figure \ref{fig:resultafterisotopy}, which is equivalent to the braid corresponding to the left hand side.

Thus, the relation holds.
\end{proof}

\begin{figure}
\begin{minipage}{.5\textwidth}
\centering
\includegraphics[height=2.5cm]{arbitraryouterboundaryparalleltwist.pdf}
\caption{}
\label{fig:outerboundarytwistbraid}
\end{minipage}
\hspace{.75cm}
\begin{minipage}{.3\textwidth}
\centering
\includegraphics[height=3cm]{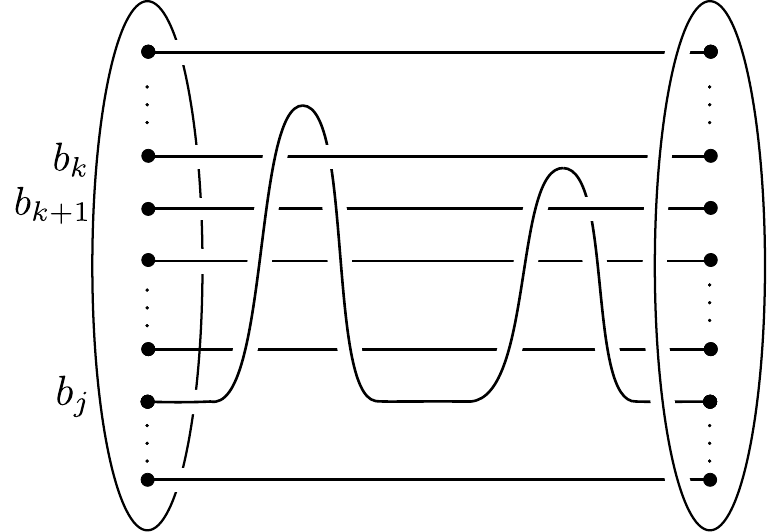}
\caption{}
\label{fig:firstbraid}
\end{minipage}
\end{figure}

\begin{figure}
\begin{minipage}{0.3\textwidth}
        \centering
        \includegraphics[height=3cm]{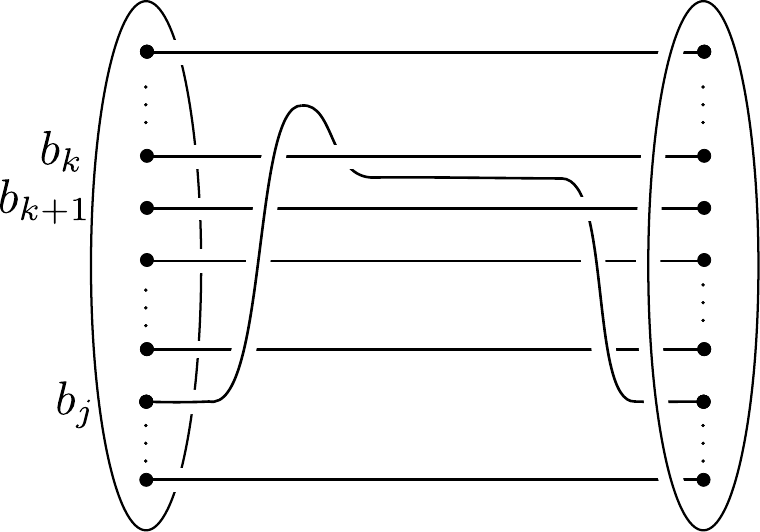}
        \caption{}
        \label{fig:midisotopy}
        \end{minipage}
         \begin{minipage}{0.26\textwidth}
        \centering
        \includegraphics[height=3cm]{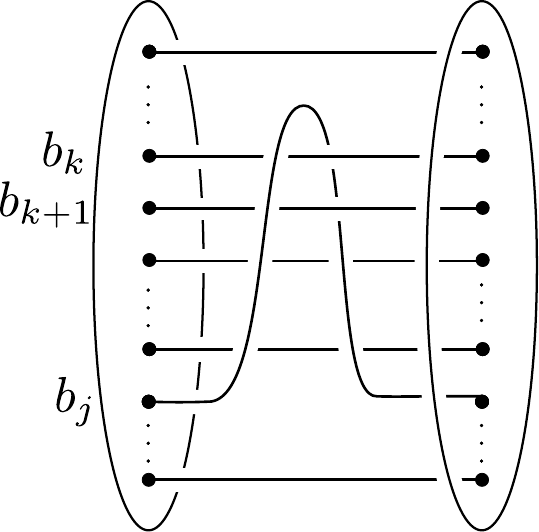}
        \caption{}
        \label{fig:endisotopy}
    \end{minipage}
\end{figure}

\begin{figure}
        \begin{minipage}{0.4\textwidth}
        \centering
        \includegraphics[height=3cm]{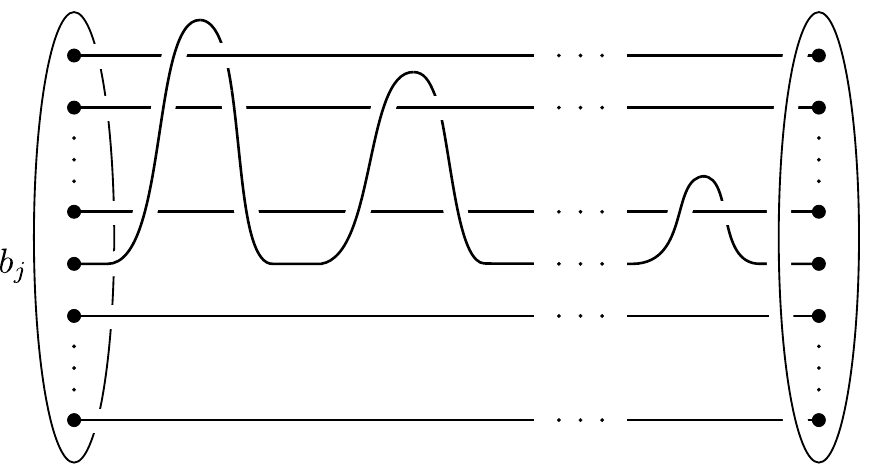} 
        \caption{}
        \label{fig:uncancelledtwists}
        \end{minipage}
         \begin{minipage}{0.3\textwidth}
        \centering
        \includegraphics[height=3cm]{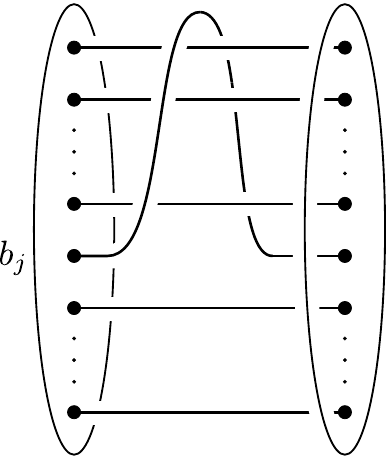}
        \caption{} 
        \label{fig:cancelledtwists}
    \end{minipage}
\end{figure}

\begin{figure}
 \begin{minipage}{.5\textwidth}
    \includegraphics[height=3cm]{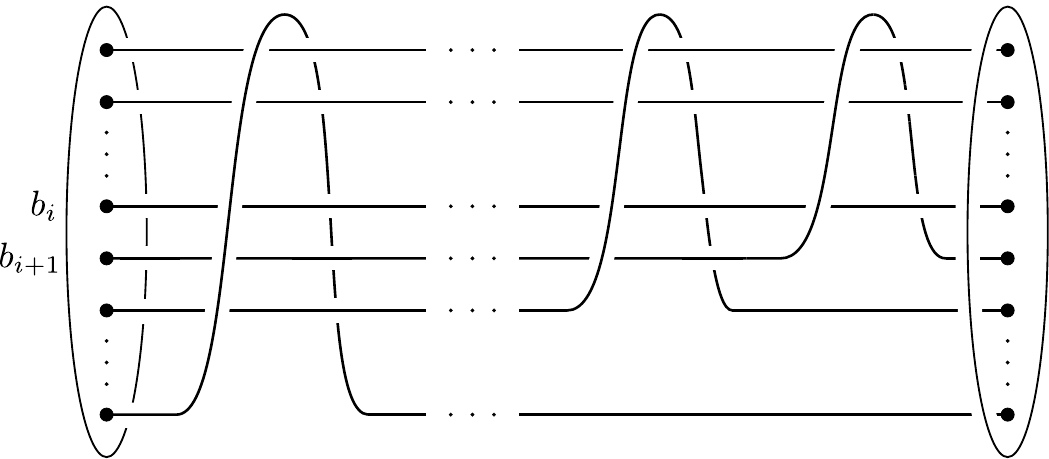} 
\caption{}
\label{fig:firstnminusitwists}
    \end{minipage}
 \begin{minipage}{.3\textwidth}
    \includegraphics[height=3cm]{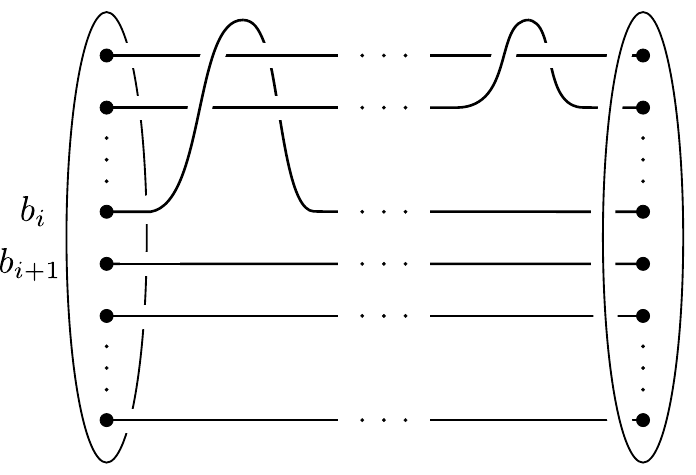} 
\caption{}
\label{fig:lastitwists}
    \end{minipage}
\end{figure}
 
\begin{figure}
\includegraphics[height=3cm]{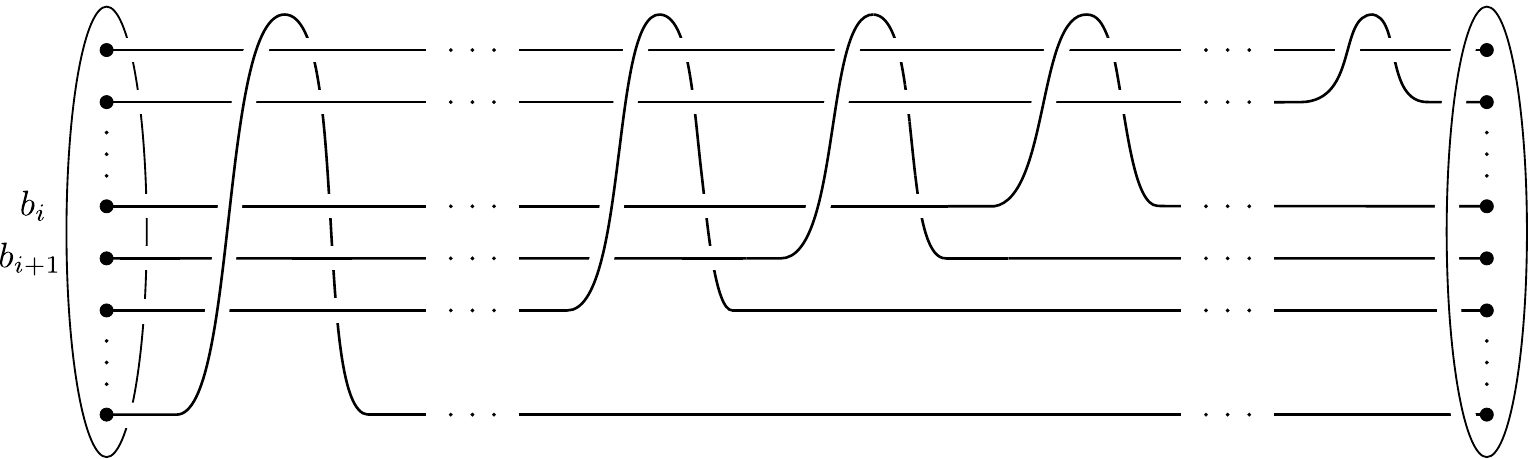} 
\caption{}
\label{fig:resultafterisotopy}
\end{figure}

	\begin{cor}
	\label{cor:generalrelations}
		A symplectic plumbing with the following graph
		
		\

	\begin{center}
\includegraphics[height=2cm]{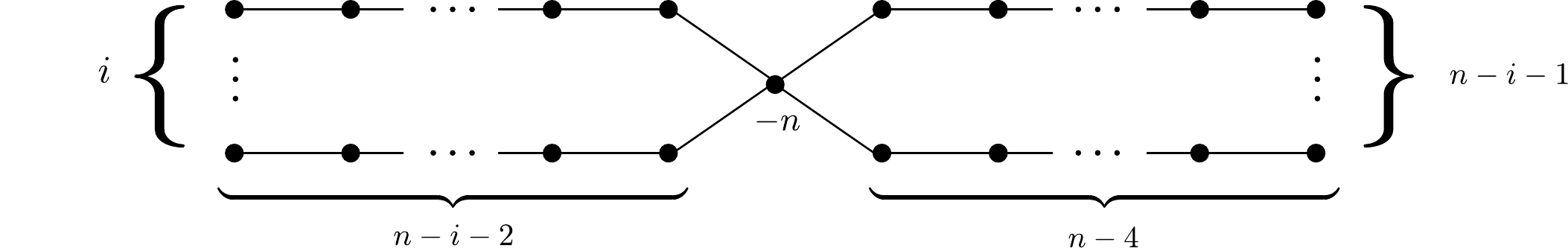} 
	\end{center}
		
		has Euler characteristic
		$$n^2-5n+6+2i-i^2$$
		and can be replaced by a different symplectic filling with the same contact boundary whose Euler characteristic is
		$$3-n+(n-i-1)(i-1)+\frac{(n-i-1)(n-i)}{2}$$
	\end{cor}

	\begin{proof}
	
	Consider the product of twists on the left hand side of the relation given in Theorem \ref{thm:relationgeneral}: $$[T_{b_1}^{n-i-1}T_{b_2}^{n-i-1}\cdots T_{b_i}^{n-i-1}][T_{b_{i+1}}^{n-3}\cdots T_{b_{n-1}}^{n-3}]T_{b_n}$$
	
	We see that, when considering the curves over which we are twisting and the corresponding fiber after all vanishing cycles have been compressed to a point, there is a sphere corresponding to the interior of $S_{0,0}^n$. It is connected to $n$ spheres, each corresponding to a curve parallel to a boundary component. Thus, in the plumbing graph there is a vertex with a value of $-n$.
	
	We note that, for a boundary component $b_j$ for which the product of twists involved $T_{b_j}^m$, the fiber with all vanishing cycles collapsed will contain $m-1$ spheres and one disk corresponding to that particular boundary component. Thus, in the plumbing graph, $T_{b_j}^m$ will result in a branch connected to the first vertex with $m-1$ vertices, each with a value of $-2$.
	
	Therefore, for each boundary component $b_j$ such that $T_{b_j}^{n-i-1}$ is included, we will have a branch from the first vertex with $n-i-2$ vertices. Note that there are exactly $i$ of these boundary components, so we will have $i$ such branches. Similarly, we will have $n-i-1$ branches with $n-4$ vertices each.
	
	Therefore, we see that the left hand side of the relation corresponds to the plumbing graph. We will now compute the Euler characteristic of the 4-manifold built by a Lefschetz fibration using this product of Dehn twists.
		
		The Euler characteristic of a 4-manifold $X$ given by the Lefschetz fibration $\pi: X\to D^2$ with general fiber $F$ is equal to $\chi(D^2)\chi(F)+k$, where $k$ is the number of vanishing cycles in the fiber \cite{OzbagciStipsicz}. Recall that, when defining a Lefschetz fibration using a product of positive Dehn twists, each Dehn twist gives a distinct vanishing cycle in the fiber.
		
		Furthermore, the Euler characteristic of the surface $S_{g,b}^n$ is $2-2g-(b+n)$ \cite{FarbMargalit}. Thus, $\chi(S_{0,0}^n)=2-2(0)-(0+n)=2-n$ and $\chi(D^2)=2-2(0)-(1+0)=2-1=1$, so $\chi(D^2)\chi(S_{0,0}^n)=2-n$
		
		Therefore, the Euler characteristic of the symplectic filling corresponding to a product of Dehn twists on $S_{0,0}^n$ is equal to $\chi(D^2)\chi(S_{0,0}^n)+k=2-n+k$, where $k$ is the number of Dehn twists. 
		
		Furthermore, we note that, for $1\leq j\leq i$, $T_{b_j}$ has an exponent of $n-i-1$. Since there are $i$ total boundary components in this collection, we see that there are $i(n-i-1)$ Dehn twists in the first set of brackets. Similarly, there are $n-i-1$ boundary components $b_j$ for which $T_{b_j}$ has an exponent of $n-3$ and finally, $T_{b_n}$ has an exponent of 1. Therefore, in total there are $i(n-i-1)+(n-i-1)(n-3)+1$ Dehn twists on the left hand side.
		
		Thus, the Euler characteristic of the resulting filling is $$2-n+[i(n-i-1)+(n-i-1)(n-3)+1]=n^2-5n+6+2i-i^2$$
		
		For the right hand side of the relation, we see that there is one Dehn twist outside of the brackets. For each set of twists inside brackets $[T_{(b_j,b_{j-1})}\cdots T_{(b_j,b_1)}]$, we see that there are $j-1$ twists. Thus, the total number of twists is $$1+\sum_{j=i+1}^{n-1}(j-1)$$ 
		
		Adjusting the index on this sum and evaluating the sum gives the result
		
		$$1+\sum_{j=1}^{n-i-1}(j+i-1)=1+\sum_{j=1}^{n-i-1}(i-1)+\sum_{j=1}^{n-i-1}j=1+(n-i-1)(i-1)+\frac{(n-i-1)(n-i)}{2}$$
		
		Therefore, the Euler characteristic of the associated filling is 
		
		$$2-n+\left[1+(n-i-1)(i-1)+\frac{(n-i-1)(n-i)}{2}\right]=3-n+(n-i-1)(i-1)+\frac{(n-i-1)(n-i)}{2}$$
		
		as desired.
	\end{proof}

	\subsection{Relations for \texorpdfstring{$S^n_{0,0}$}{Sn0,0} for small values of \texorpdfstring{$n$}{n}}

We skip the cases where $n<5$ because a full set of relations has been defined for $S_{0,0}^n$ for each $n<5$. In particular, for $n=2$ and $n=3$, $S_{0,0}^n$ contains only boundary parallel curves, and the Lantern Relation, which is defined for $S_{0,0}^4$, completely describes the relations between elements of its mapping class monoid of the form we are considering. For the values $n=5,6,7$, we will present relations on the mapping class monoid of $S_{0,0}^n$ and discuss the implications of these relations on associated plumbing graphs and the Euler characteristics of the resulting fillings.

\begin{prop}
\label{prop:5bdrelations}
		The following are relations in the mapping class monoid of $S^5_{0,0}$:
		\begin{enumerate}
			\item \label{rel51} $T_{b_1}^2T_{b_2}^2T_{b_3}^2T_{b_4}^2T_{b_5}=[T_{b_1,b_2}][T_{b_2,b_3}T_{b_1,b_3}][T_{b_3,b_4}T_{b_2,b_4}T_{b_1,b_4}]$
			\item \label{rel52} $T_{b_1}T_{b_2}T_{b_3}T_{b_4}^2T_{b_5}=[T_{b_3,b_4}T_{b_2,b_4}T_{b_1,b_4}]T_{b_1,b_2,b_3}$
		\end{enumerate}
	\end{prop}

	\begin{proof}
		Note that relation~\ref{rel51} follows directly from Theorem \ref{thm:relationgeneral} by taking $n=5$ and $i=2$. Similarly, \ref{rel52} follows  from Theorem \ref{thm:relationgeneral} by taking $n=5$ and $i=3$.
\end{proof}

\begin{cor}   
\hspace{2em}
\begin{itemize}
\item[\normalfont{i.}]
		A symplectic plumbing with the graph in Figure \ref{fig:5bdgraph1} has Euler characteristic 3 and can be replaced by a different symplectic filling with the same contact boundary whose Euler characteristic is 3.
		
		\item[\normalfont{ii.}] A symplectic plumbing with the graph in Figure \ref{fig:5bdgraph2} has Euler characteristic 6 and can be replaced by a different symplectic filling with the same contact boundary whose Euler characteristic is 1.\end{itemize}
	\end{cor}

\begin{center}

\begin{figure}[!htb]
    \centering
    \begin{minipage}{.3\textwidth}
        \centering
        \includegraphics[height=2cm]{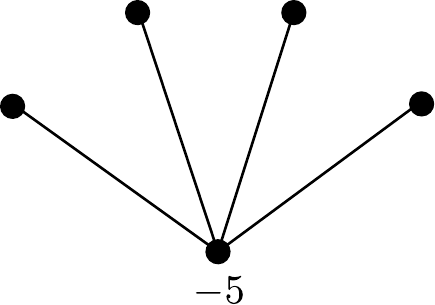}
        \caption{}
        \label{fig:5bdgraph1}
    \end{minipage}
    \begin{minipage}{0.3\textwidth}
        \centering
        \includegraphics[height=2cm]{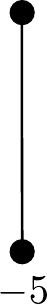}
        \caption{}
        \label{fig:5bdgraph2}
    \end{minipage}
\end{figure}

\end{center}

Note that both of these statements follow directly from Corollary \ref{cor:generalrelations}.
	
\begin{prop}
\label{prop:6bdrelations}
		The following are relations in the mapping class monoid of $S^6_{0,0}$:
		
		\
		
		\begin{enumerate}
			\item \label{rel61} $T_{b_1}^3T_{b_2}^3T_{b_3}^3T_{b_4}^3T_{b_5}^3T_{b_6}=T_{b_2,b_1}[T_{b_3,b_2}T_{b_3,b_1}][T_{b_4,b_3}T_{b_4,b_2}T_{b_4,b_1}][T_{b_5,b_4}T_{b_5,b_3}T_{b_5,b_2}T_{b_5,b_1}]$
			
			\
			
			\item \label{rel62} $T_{b_1}^2T_{b_2}^2T_{b_3}^2T_{b_4}^3T_{b_5}^3T_{b_6}=T_{b_1,b_2,b_3}[T_{b_4,b_3}T_{b_4,b_2}T_{b_4,b_1}][T_{b_5,b_4}T_{b_5,b_3}T_{b_5,b_2}T_{b_5,b_1}]$
			
			\
			
			\item \label{rel63}
			$ T_{b_1}T_{b_2}T_{b_3}T_{b_4}T_{b_5}^3T_{b_6}=T_{b_1,b_2,b_3,b_4}[T_{b_5,b_4}T_{b_5,b_3}T_{b_5,b_2}T_{b_5,b_1}]$
			
			\
			
			\item \label{rel64}
			$T_{b_1}^2T_{b_2}T_{b_3}^2T_{b_4}^2T_{b_5}^2T_{b_6}=T_{b_3,b_4}T_{b_1,b_2,b_4}T_{b_4,b_5}T_{b_2,b_3,b_5}T_{b_1,b_5}T_{b_1,b_3}$
			
			\
			
			\item \label{rel65} $T_{b_1}^2T_{b_2}^2T_{b_3}^2T_{b_4}T_{b_5}^2T_{b_6}=T_{b_2,b_3}T_{b_1,b_3}T_{b_3,b_4,b_5}T_{b_2,b_5}T_{b_1,b_5}T_{b_1,b_2,b_4}$
			
			\
			
			\item \label{rel66} $T_{b_1}^2T_{b_2}^2T_{b_3}T_{b_4}^2T_{b_5}^2T_{b_6}=T_{b_1,b_2,b_3}T_{b_3,b_4,b_5}T_{b_2,b_5}T_{b_1,b_5}T_{b_2,b_4}T_{b_1,b_4}$
			
			\
			
			\item \label{rel67} $T_{b_1}^2T_{b_2}^2T_{b_3}^3T_{b_4}^2T_{b_5}^3T_{b_6}=T_{b_2,b_3}T_{b_1,b_3}T_{b_3,b_4}T_{b_4,b_5}T_{b_3,b_5}T_{b_2,b_5}T_{b_1,b_5}T_{b_1,b_2,b_4}$
			
		\end{enumerate}
	\end{prop}
	
	\

	\begin{proof}
		Note that relations 1, 2, and 3 follow directly from Theorem \ref{thm:relationgeneral}.
		
		Now we prove the remaining relations. The braid corresponding to the left hand side of each relation is shown in Figure \ref{fig:braid6}.
		
		\begin{figure}
		\centering

		\includegraphics[height=2cm]{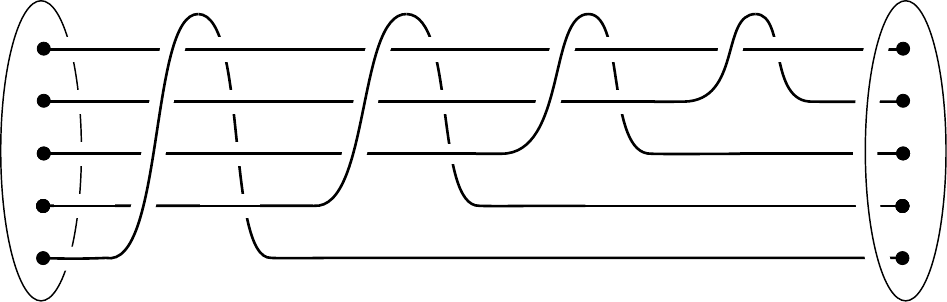} 
\caption{ }
\label{fig:braid6}
	\end{figure}

Figure \ref{fig:braid64isotopies} shows the braid corresponding to the right hand side of relation \ref{rel64} and two steps in the isotopy between that braid and the braid corresponding to the left hand side of the relation. Since the two braids are isotopic and each boundary component has the same multiplicity under the products of twists on both sides, we conclude that \ref{rel64} holds.

Similarly, note that the multiplicity of each boundary component is the same under both sides of relations \ref{rel65}, \ref{rel66}, and \ref{rel67}. Furthermore, Figures \ref{fig:braidrel65}, \ref{fig:braidrel66}, and \ref{fig:braidrel67} show the braid corresponding to the right hand side of relations 5, 6, and 7, respectively, and demonstrate some of the steps in each isotopy between the braids corresponding to the right and left sides of each relation.

Thus, for relations 5, 6, and 7, the multiplicity of each boundary component is the same on both sides and the corresponding braids for each side are isotopic, so these relations hold.
\end{proof}	

	\begin{figure}
\centering
\includegraphics[height=2cm]{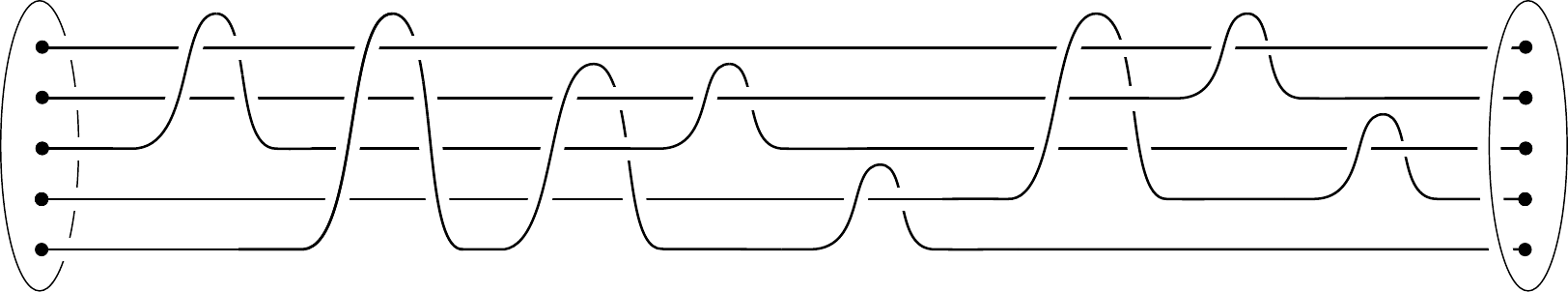}

\vspace{.25cm}

\includegraphics[height=1.5cm]{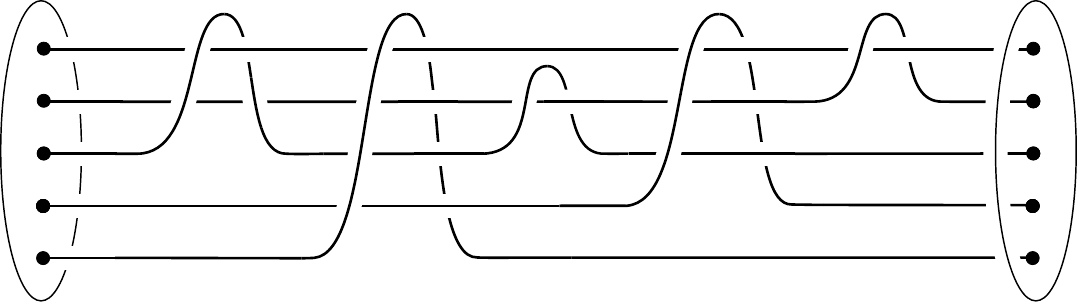} 
\hspace{1cm}
\includegraphics[height=1.5cm]{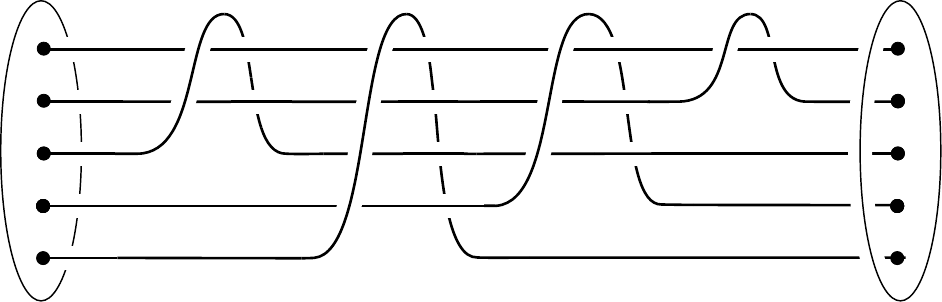}
\caption{ }
\label{fig:braid64isotopies}
	\end{figure}

	\begin{figure}
\centering
\includegraphics[height=2cm]{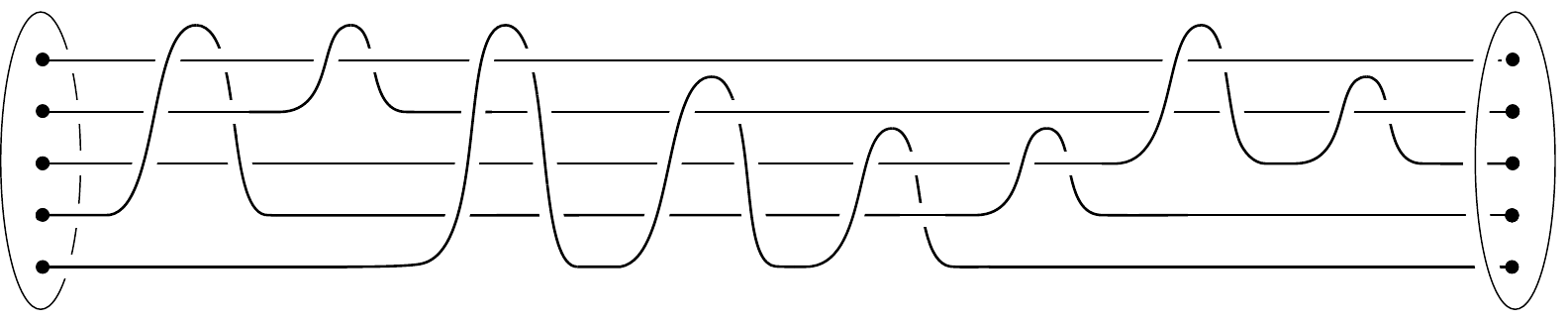}
 
\vspace{.25cm}

\includegraphics[height=1.5cm]{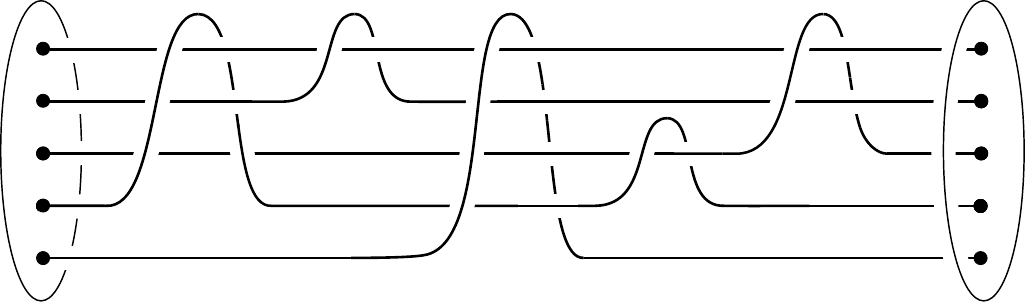}
\hspace{1cm}
\includegraphics[height=1.5cm]{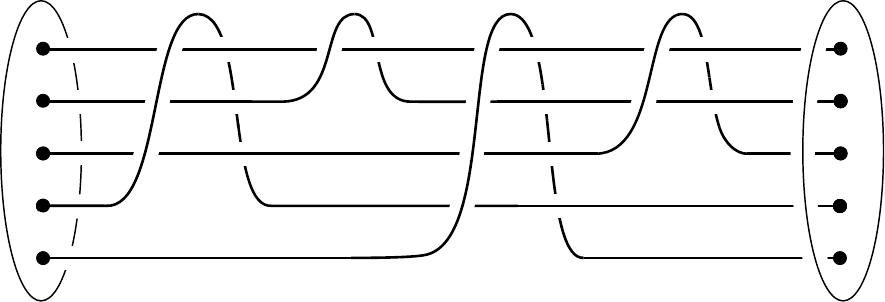} 
\caption{ }
\label{fig:braidrel65}
	\end{figure}
	
	\begin{figure}
\centering
\includegraphics[height=2cm]{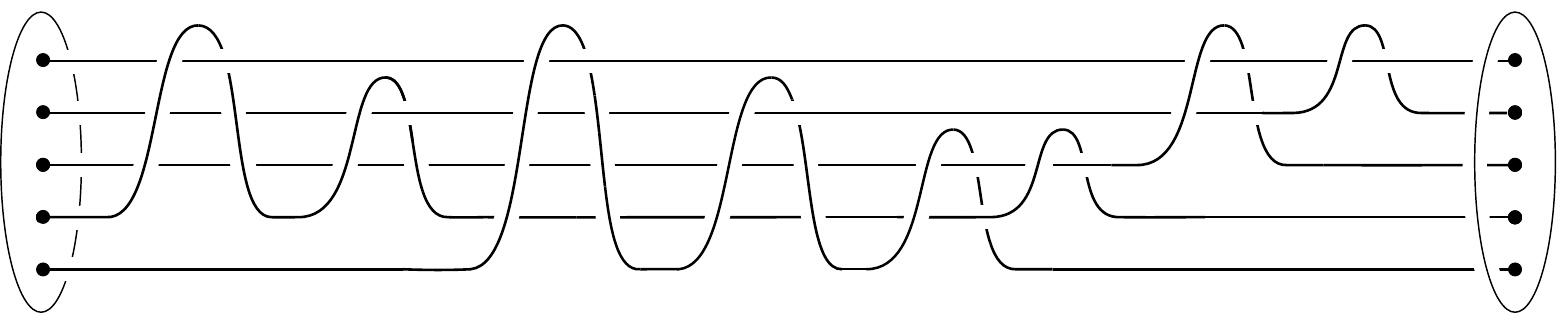}

\vspace{.25cm}

\includegraphics[height=1.5cm]{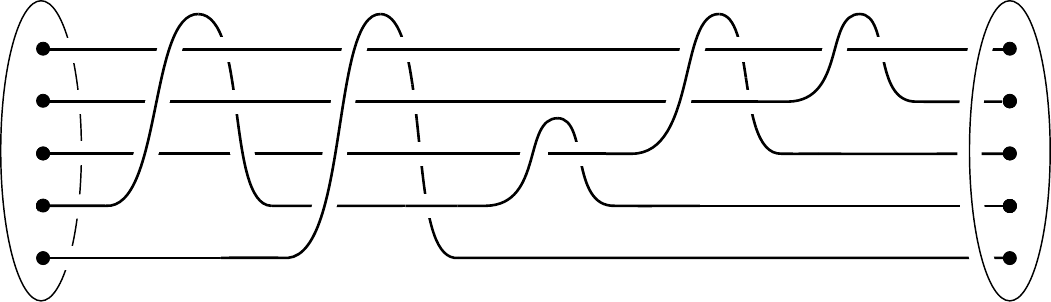}
\hspace{1cm}
\includegraphics[height=1.5cm]{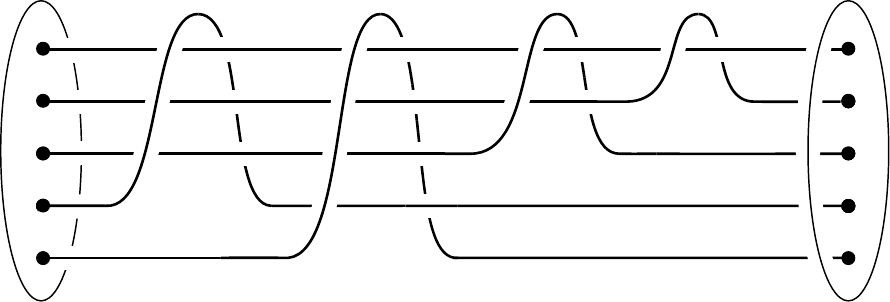}
\caption{ }
\label{fig:braidrel66}
	\end{figure}

	\begin{figure}
\includegraphics[height=2cm]{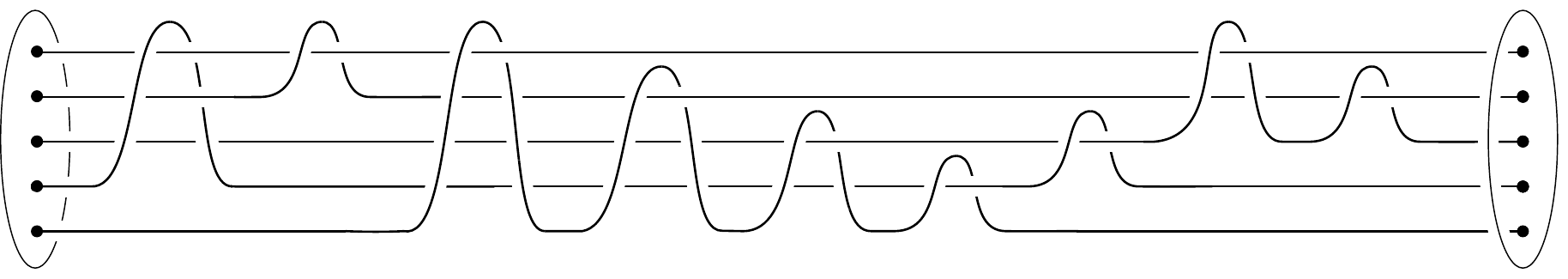} 

\vspace{.25cm}

\includegraphics[height=1.5cm]{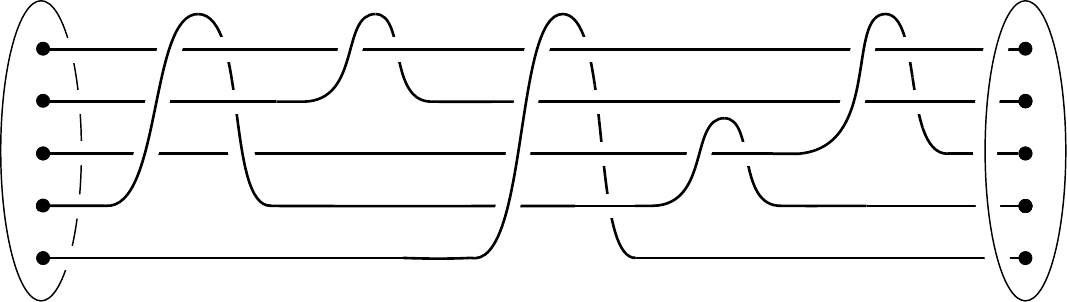} 
\hspace{1cm}
\includegraphics[height=1.5cm]{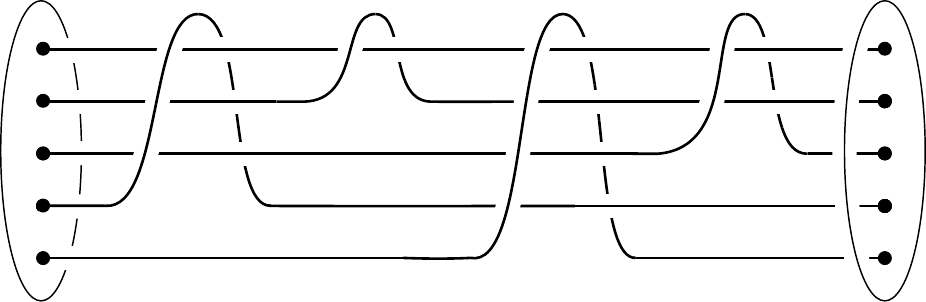}
\caption{ }
\label{fig:braidrel67}
	\end{figure}

Determining the plumbing and Euler characteristics of the Lefschetz fibrations corresponding to these relations as in the proof of Corollary \ref{cor:generalrelations}, we obtain the following corollary.

\begin{cor}   
\hspace{2em}
\begin{itemize}
\item[\normalfont{i.}]
		A symplectic plumbing with the graph in Figure \ref{fig:6bdgraph2} has Euler characteristic 12 and can be replaced by a different symplectic filling with the same contact boundary whose Euler characteristic is 6.
		
		\item[\normalfont{ii.}] A symplectic plumbing with the graph in Figure \ref{fig:6bdgraph3} has Euler characteristic 9 and can be replaced by a different symplectic filling with the same contact boundary whose Euler characteristic is 4.
		
		\item[\normalfont{iii.}] A symplectic plumbing with the graph in Figure \ref{fig:6bdgraph4} has Euler characteristic 4 and can be replaced by a different symplectic filling with the same contact boundary whose Euler characteristic is 1.
		
		\item[\normalfont{iv.}] A symplectic plumbing with the graph in Figure \ref{fig:6bdgraph1} has Euler characteristic 6 and can be replaced by a different symplectic filling with the same contact boundary whose Euler characteristic is 2.
		\end{itemize}
	\end{cor}

	\begin{figure}[!htb]
    \centering
    \begin{minipage}{.26\textwidth}
        \centering
        \includegraphics[height=2cm]{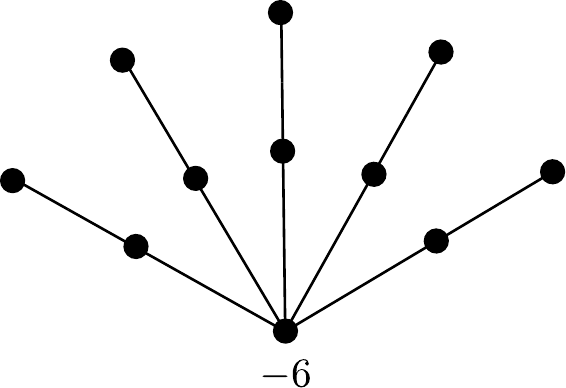}
        \caption{}
        \label{fig:6bdgraph2}
    \end{minipage}%
    \begin{minipage}{0.26\textwidth}
        \centering
        \includegraphics[height=2cm]{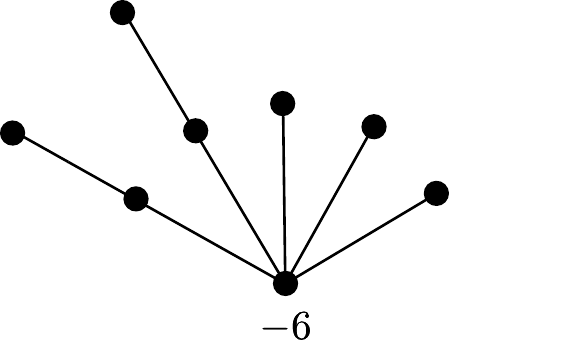}
        \caption{}
        \label{fig:6bdgraph3}
    \end{minipage}%
    \begin{minipage}{0.26\textwidth}
        \centering
        \includegraphics[height=2cm]{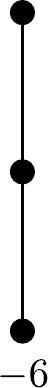}
        \caption{}
        \label{fig:6bdgraph4}
        \end{minipage}%
        \begin{minipage}{0.26\textwidth}
        \centering
        \includegraphics[height=2cm]{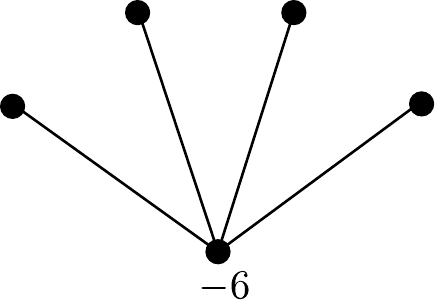}
        \caption{}
        \label{fig:6bdgraph1}
        \end{minipage}%
\end{figure}

	\begin{prop}
	\label{prop:7bdrelations}
		The following are relations in the mapping class monoid of $S^7_{0,0}$:
		\begin{enumerate}
			\item \label{rel71} $$T_{b_1}^4T_{b_2}^4T_{b_3}^4T_{b_4}^4T_{b_5}^4T_{b_5}^4T_{b_7}=T_{b_1,b_2}[T_{b_3,b_2}T_{b_3,b_1}][T_{b_4,b_3}T_{b_4,b_2}T_{b_4,b_1}][T_{b_5,b_4}T_{b_5,b_3}T_{b_5,b_2}T_{b_5,b_1}][T_{b_6,b_5}T_{b_6,b_4}T_{b_6,b_3}T_{b_6,b_2}T_{b_6,b_1}]$$
			
\			
			
			\item \label{rel72} $$T_{b_1}T_{b_2}T_{b_3}T_{b_4}T_{b_5}T_{b_6}^4T_{b_7}=T_{b_1,b_2,b_3,b_4,b_5}[T_{b_6,b_5}T_{b_6,b_4}T_{b_6,b_3}T_{b_6,b_2}T_{b_6,b_1}]$$
			
\
			
			\item \label{rel73} $$T_{b_1}^3T_{b_2}^3T_{b_3}^3T_{b_4}^4T_{b_5}^4T_{b_6}^4T_{b_7}=T_{b_1,b_2,b_3}[T_{b_4,b_3}T_{b_4,b_2}T_{b_4,b_1}][T_{b_5,b_4}T_{b_5,b_3}T_{b_5,b_2}T_{b_5,b_1}][T_{b_6,b_5}T_{b_6,b_4}T_{b_6,b_3}T_{b_6,b_2}T_{b_6,b_1}]$$
			
\
			
\item \label{rel74} $$T_{b_1}^2T_{b_2}^2T_{b_3}^2T_{b_4}^2T_{b_5}^4T_{b_6}^4T_{b_7}=T_{b_1,b_2,b_3,b_4}[T_{b_5,b_4}T_{b_5,b_3}T_{b_5,b_2}T_{b_5,b_1}][T_{b_6,b_5}T_{b_6,b_4}T_{b_6,b_3}T_{b_6,b_2}T_{b_6,b_1}]$$ 
			
\
			
\item \label{rel75} $T_{b_1}^2T_{b_2}^2T_{b_3}^2T_{b_4}T_{b_5}^3T_{b_6}^3T_{b_7}=T_{b_1,b_2,b_3,b_4}T_{b_4,b_5,b_6}T_{b_3,b_6}T_{b_2,b_6}T_{b_1,b_6}T_{b_3,b_5}T_{b_2,b_5}T_{b_1,b_5}$

\
			
\item \label{rel76} $T_{b_1}^3T_{b_2}^3T_{b_3}^2T_{b_4}T_{b_5}^2T_{b_6}^2T_{b_7}=T_{b_2,b_3}T_{b_1,b_3}T_{b_3,b_4,b_5,b_6}T_{b_2,b_6}T_{b_1,b_6}T_{b_2,b_5}T_{b_1,b_5}T_{b_1,b_2,b_4}$	

\

\item \label{rel77} $T_{b_1}^2T_{b_2}^2T_{b_3}T_{b_4}^3T_{b_5}^2T_{b_6}^3T_{b_7}=T_{b_4,b_5}T_{b_1,b_2,b_3,b_5}T_{b_5,b_6}T_{b_3,b_4,b_6}T_{b_2,b_6}T_{b_1,b_6}T_{b_2,b_4}T_{b_1,b_4}$

\

\item \label{rel78} $T_{b_1}^2T_{b_2}^2T_{b_3}^2T_{b_4}^3T_{b_5}T_{b_6}^3T_{b_7}=T_{b_4,b_3}T_{b_4,b_2}T_{b_4,b_1}T_{b_4,b_5,b_6}T_{b_6,b_3}T_{b_6,b_2}T_{b_6,b_1}T_{b_1,b_2,b_3,b_5}$

\

\item \label{rel79} $T_{b_1}^3T_{b_2}^2T_{b_3}^3T_{b_4}^3T_{b_5}^3T_{b_6}^4T_{b_7}=T_{b_3,b_4}T_{b_1,b_2,b_4}T_{b_4,b_5}T_{b_2,b_3,b_5}T_{b_1,b_5}T_{b_5,b_6}T_{b_4,b_6}T_{b_3,b_6}T_{b_2,b_6}T_{b_1,b_6}T_{b_1,b_3}$

\

\item \label{rel710} $T_{b_1}^3T_{b_2}^2T_{b_3}^3T_{b_4}^3T_{b_5}^3T_{b_6}^4T_{b_7}=T_{b_4,b_3}T_{b_1,b_2,b_4}T_{b_5,b_4}T_{b_5,b_3}T_{b_5,b_2}T_{b_5,b_1}T_{b_5,b_6}T_{b_4,b_6}T_{b_2,b_3,b_6}T_{b_1,b_6}T_{b_1,b_3}$

\

\item \label{rel711} $T_{b_1}^3T_{b_2}^3T_{b_3}^3T_{b_4}^2T_{b_5}^3T_{b_6}^4T_{b_7}=T_{b_2,b_3}T_{b_1,b_3}T_{b_3,b_4,b_5}T_{b_2,b_5}T_{b_1,b_5}T_{b_5,b_6}T_{b_1,b_2,b_4}T_{b_4,b_6}T_{b_3,b_6}T_{b_2,b_6}T_{b_1,b_6}$

\

\item \label{rel712} $T_{b_1}^3T_{b_2}^3T_{b_3}^4T_{b_4}^2T_{b_5}^3T_{b_6}^3T_{b_7}=T_{b_2,b_3}T_{b_1,b_3}T_{b_3,b_4}T_{b_1,b_2,b_4}T_{b_4,b_5,b_6}T_{b_3,b_6}T_{b_2,b_6}T_{b_1,b_6}T_{b_3,b_5}T_{b_2,b_5}T_{b_1,b_5}$

\
 
\item \label{rel713}

$T_{b_1}^3T_{b_2}^3T_{b_3}^2T_{b_4}^3T_{b_5}^4T_{b_6}^3T_{b_7}=T_{b_4,b_5}T_{b_3,b_5}T_{b_2,b_5}T_{b_1,b_5}T_{b_5,b_6}T_{b_3,b_4,b_5}T_{b_2,b_6}T_{b_1,b_6}T_{b_2,b_4}T_{b_1,b_4}T_{b_1b_2b_3}$

\

\item \label{rel714} $T_{b_1}^3T_{b_2}^3T_{b_3}^3T_{b_4}^3T_{b_5}^3T_{b_6}^3T_{b_7}=T_{b_3,b_2}T_{b_3,b_1}T_{b_5,b_4}T_{b_5,b_3}T_{b_1,b_2,b_5}T_{b_5,b_6}T_{b_3,b_4,b_6}T_{b_6,b_2}T_{b_6,b_1}T_{b_4,b_2}T_{b_4,b_1}$

\

\item \label{rel715} $T_{b_1}^4T_{b_2}^3T_{b_3}^2T_{b_4}^3T_{b_5}^2T_{b_6}^4T_{b_7}=T_{b_2,b_3,b_4}T_{b_1,b_4}T_{b_1,b_3}T_{b_1,b_2}T_{b_4,b_5,b_6}T_{b_3,b_6}T_{b_2,b_6}T_{b_1,b_6}T_{b_3,b_5}T_{b_2,b_5}T_{b_1,b_5}$

\

\item \label{rel716} $T_{b_1}^3T_{b_2}^3T_{b_3}^3T_{b_4}^3T_{b_5}^3T_{b_6}^3T_{b_7}=T_{b_1,b_2,b_3}T_{b_3,b_4}T_{b_2,b_4}T_{b_1,b_4}T_{b_4,b_5,b_6}T_{b_3,b_6}T_{b_2,b_6}T_{b_1,b_6}T_{b_3,b_5}T_{b_2,b_5}T_{b_1,b_5}$
		
		\end{enumerate}
	\end{prop}

	\begin{proof}
		Note that relations \ref{rel71}, \ref{rel72}, \ref{rel73}, and \ref{rel74} follow directly from Theorem \ref{thm:relationgeneral}.
		
		Note that, for each of the remaining relations, the multiplicity of each boundary component is the same on both sides. Therefore, in order to prove each relation, it suffices to show that the corresponding braids are isotopic.
		
		Figure \ref{fig:braid7} shows the braid corresponding to any product of boundary twists of the form $T_{b_1}^{a_1}\cdots T_{b_6}^{a_6}T_{b_7}$. Note that the left hand side of each relation is of this form.
		
		Figures \ref{fig:braidrel75} through \ref{fig:braidrel716} show the braids corresponding to the right hand side of each relation in ascending numerical order. Each is isotopic to the braid in Figure \ref{fig:braid7}. Verification of the existence each isotopy is left to the reader.
		
From this, we conclude that each relation holds.		
\end{proof}

		 \begin{figure*}
		 \centering
		\includegraphics[height=1.5cm]{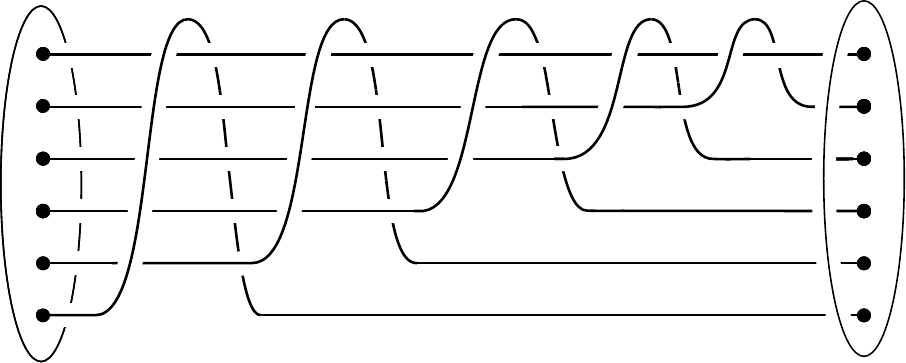}
		\caption{}
		\label{fig:braid7} 
		\end{figure*}

\begin{figure}
\centering		
\includegraphics[height=1.5cm]{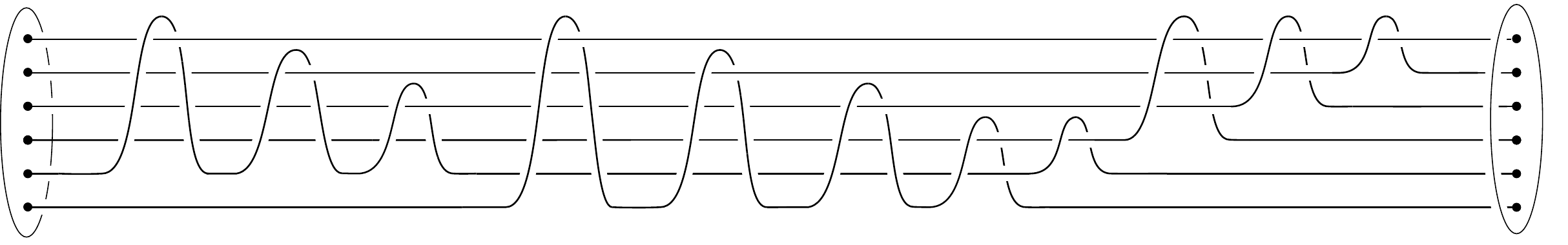}
\caption{}
\label{fig:braidrel75} 	
\end{figure}

\begin{figure}
\centering		
\includegraphics[height=1.5cm]{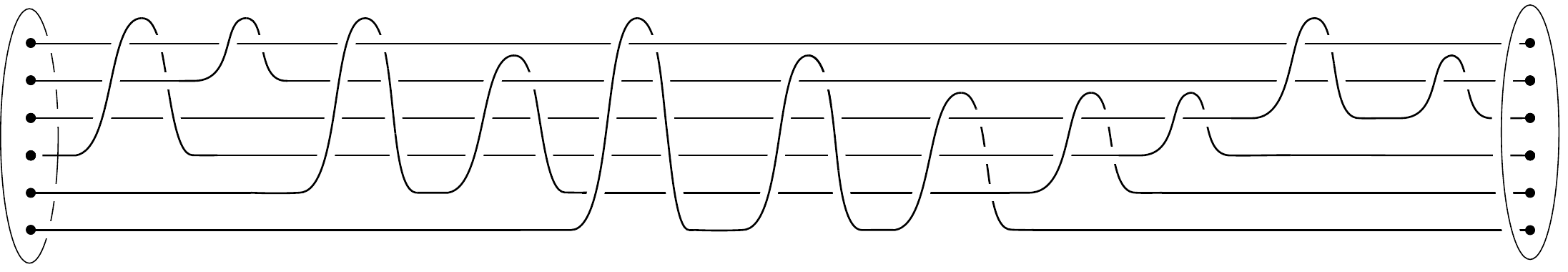} 
\caption{}
\label{fig:braidrel76}
\end{figure}

\begin{figure}
\centering		
\includegraphics[height=1.5cm]{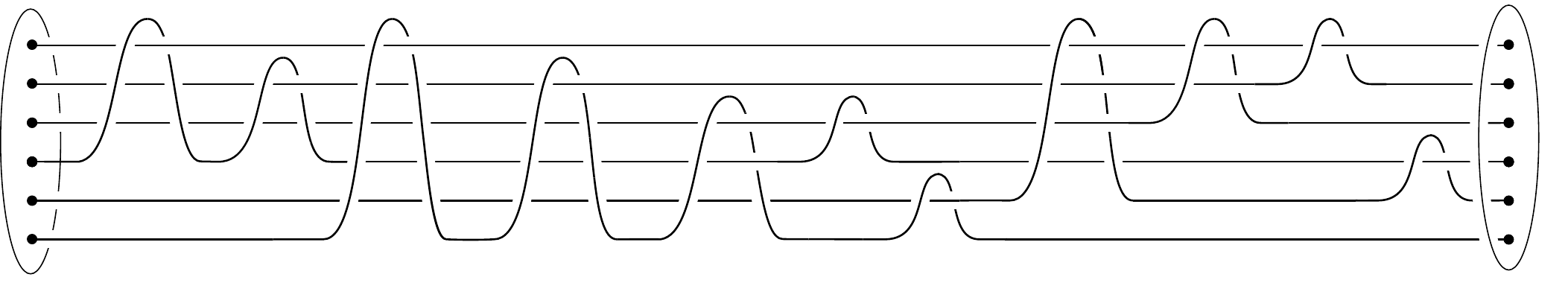}
\caption{}
\label{fig:braidrel77}
\end{figure}

\begin{figure}
\centering		
\includegraphics[height=1.5cm]{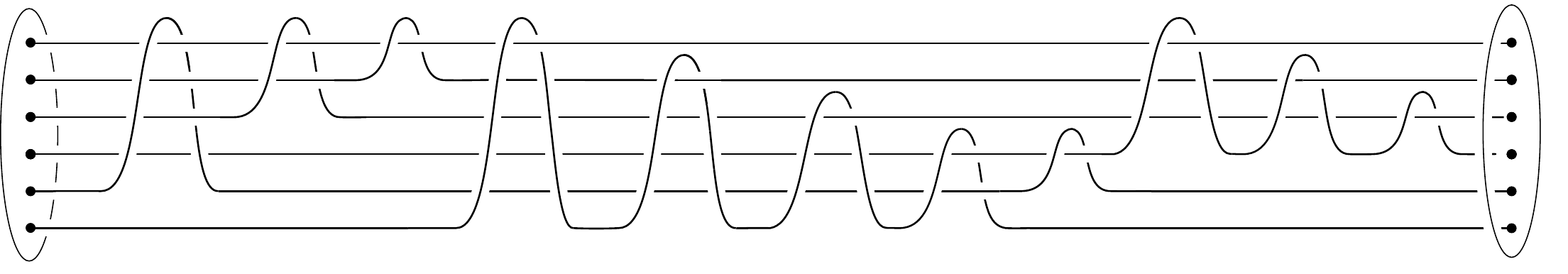}
\caption{} 
\label{fig:braidrel78}	
\end{figure}

\begin{figure}
\centering		
\includegraphics[height=1.5cm]{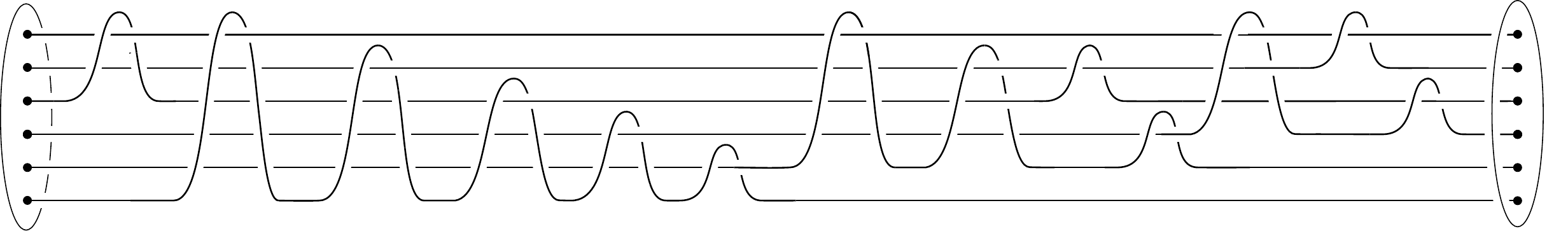}
\caption{} 
\label{fig:braidrel79}	
\end{figure}

\begin{figure}
\centering		
\includegraphics[height=1.5cm]{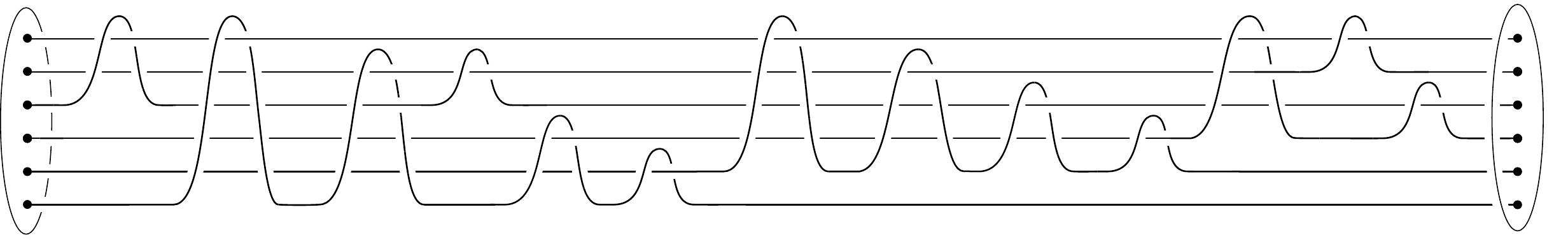}
\caption{} 
\label{fig:braidrel710}	
\end{figure}	

\begin{figure}
\centering		
\includegraphics[height=1.5cm]{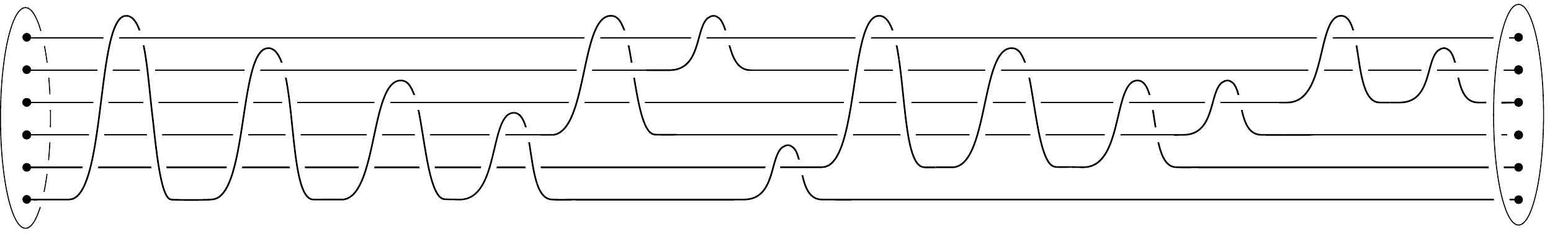}
\caption{}
\label{fig:braidrel711} 	
\end{figure}

\begin{figure}
\centering		
\includegraphics[height=1.5cm]{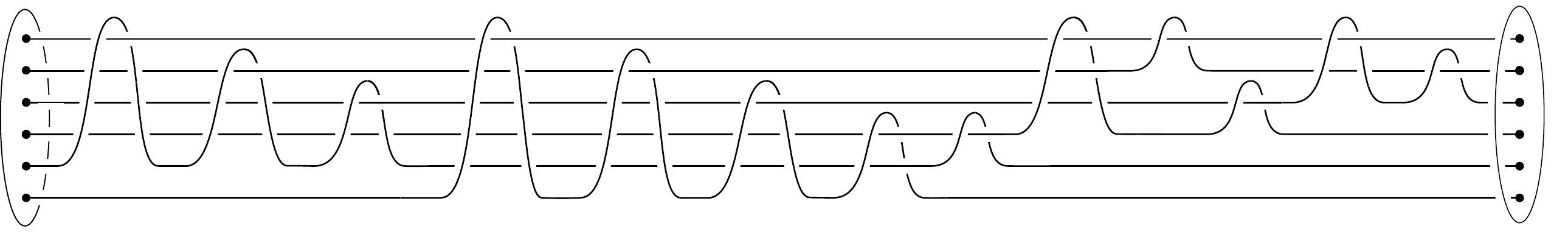}
\caption{}
\label{fig:braidrel712} 	
\end{figure}

\begin{figure}
\centering		
\includegraphics[height=1.5cm]{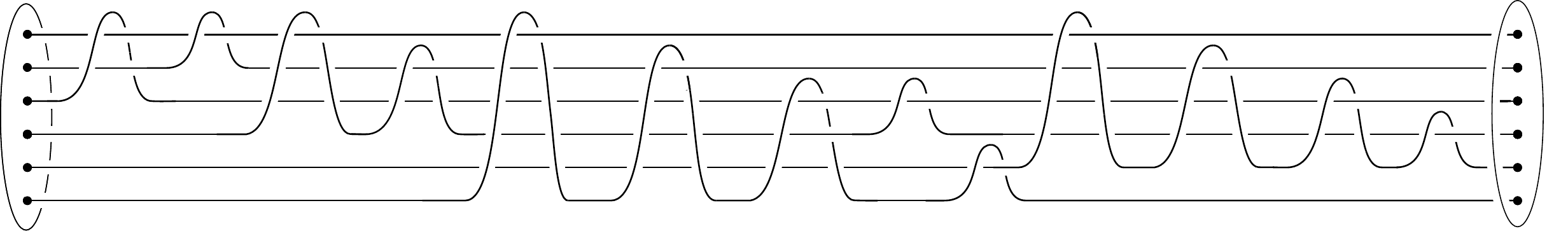}
\caption{}
\label{fig:braidrel713} 	
\end{figure}

\begin{figure}
\centering		
\includegraphics[height=1.5cm]{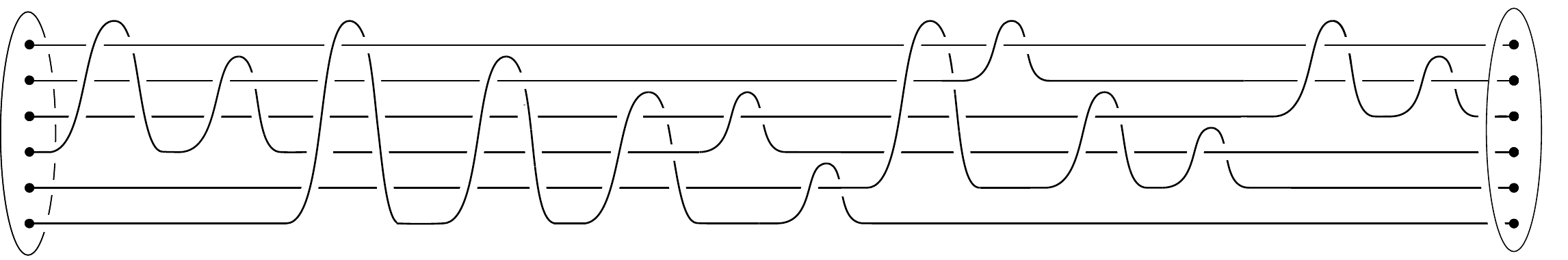}
\caption{}
\label{fig:braidrel714} 	
\end{figure}

\begin{figure}
\centering		
\includegraphics[height=1.5cm]{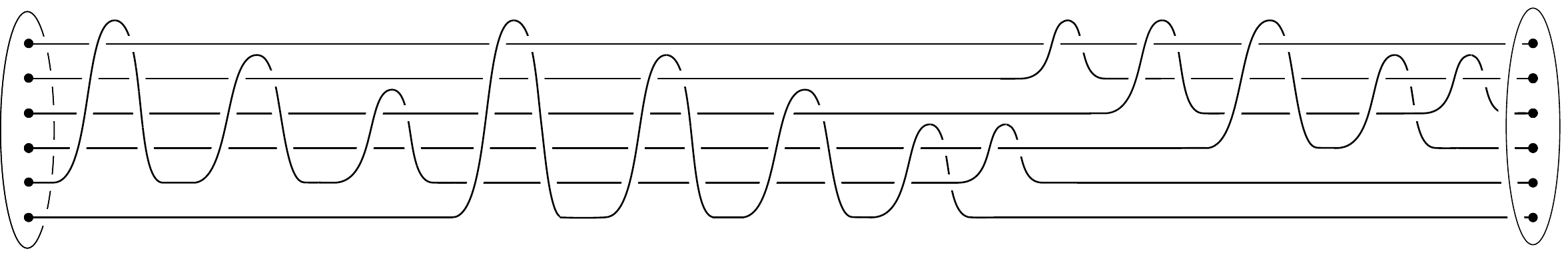}
\caption{}
\label{fig:braidrel715} 	
\end{figure}
	
\begin{figure}
\centering		
\includegraphics[height=1.5cm]{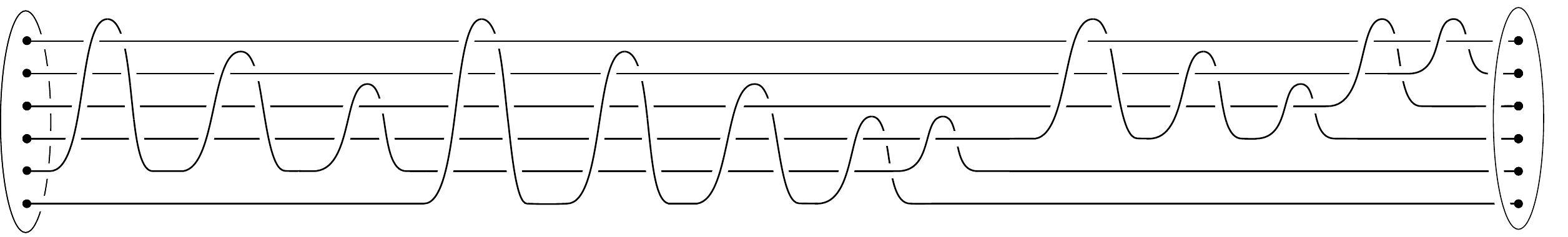}
\caption{} 
\label{fig:braidrel716}	
\end{figure}

\newpage
	
\begin{cor}   
\label{cor:7bdplumbinggraphs}
\hspace{2em}
\begin{itemize}
\item[\normalfont{i.}]
		A symplectic plumbing with the graph in Figure \ref{fig:7bdgraph1} has Euler characteristic 16 and can be replaced by a different symplectic filling with the same contact boundary whose Euler characteristic is 10.
		
		\item[\normalfont{ii.}] A symplectic plumbing with the graph in Figure \ref{fig:7bdgraph7} has Euler characteristic 5 and can be replaced by a different symplectic filling with the same contact boundary whose Euler characteristic is 1.
		
		\item[\normalfont{iii.}] A symplectic plumbing with the graph in Figure \ref{fig:7bdgraph3} has Euler characteristic 17 and can be replaced by a different symplectic filling with the same contact boundary whose Euler characteristic is 8.
		
		\item[\normalfont{iv.}] A symplectic plumbing with the graph in Figure \ref{fig:7bdgraph4} has Euler characteristic 12 and can be replaced by a different symplectic filling with the same contact boundary whose Euler characteristic is 5.
		
		\item[\normalfont{v.}] A symplectic plumbing with the graph in Figure \ref{fig:7bdgraph2} has Euler characteristic 9 and can be replaced by a different symplectic filling with the same contact boundary whose Euler characteristic is 3.
		
		\item[\normalfont{vi.}] A symplectic plumbing with the graph in Figure \ref{fig:7bdgraph5} has Euler characteristic 14 and can be replaced by a different symplectic filling with the same contact boundary whose Euler characteristic is 6.
		
		\item[\normalfont{vii.}] A symplectic plumbing with the graph in Figure \ref{fig:7bdgraph6} has Euler characteristic 14 and can be replaced by a different symplectic filling with the same contact boundary whose Euler characteristic is 6.
		\end{itemize}
	\end{cor}

	\begin{figure}[!htb]
    \centering
    \begin{minipage}{0.26\textwidth}
        \centering
        \includegraphics[height=2cm]{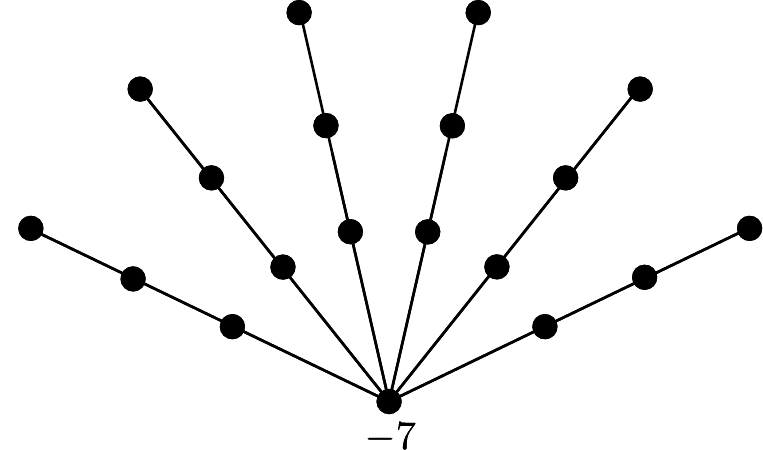}
        \caption{}
        \label{fig:7bdgraph1}
        \end{minipage}%
         \begin{minipage}{0.26\textwidth}
        \centering
        \includegraphics[height=2.5cm]{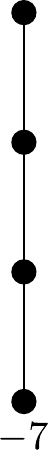}
        \caption{}
        \label{fig:7bdgraph7}
    \end{minipage}%
    \begin{minipage}{0.26\textwidth}
        \centering
        \includegraphics[height=2cm]{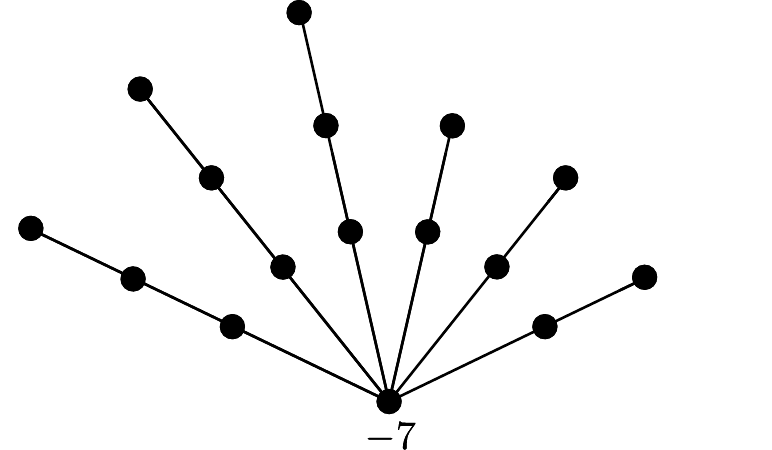}
        \caption{}
        \label{fig:7bdgraph3}
    \end{minipage}%
    \begin{minipage}{0.26\textwidth}
        \centering
        \includegraphics[height=2cm]{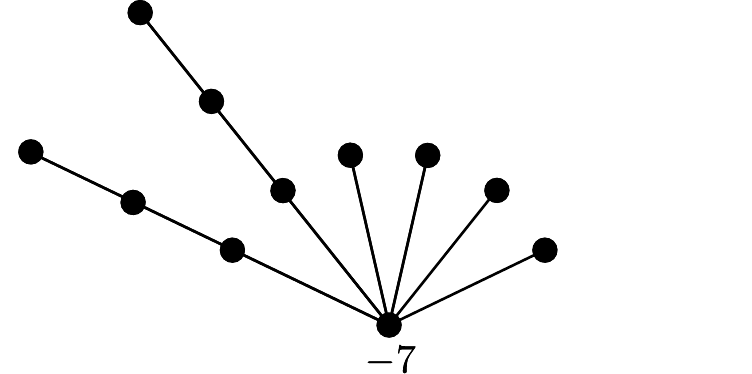}
        \caption{}
        \label{fig:7bdgraph4}
        \end{minipage}%
\end{figure}

\begin{figure}[!htb]
    \centering
    \begin{minipage}{.26\textwidth}
        \centering
        \includegraphics[height=2cm]{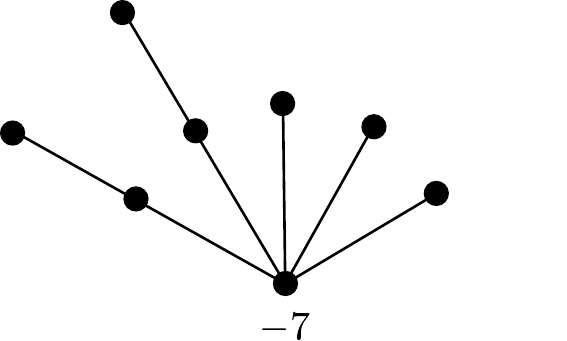}
        \caption{}
        \label{fig:7bdgraph2}
    \end{minipage}%
    \begin{minipage}{0.26\textwidth}
        \centering
        \includegraphics[height=2cm]{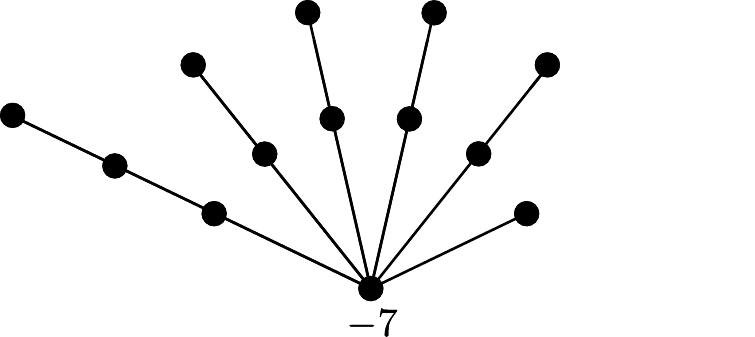}
        \caption{}
        \label{fig:7bdgraph5}
        \end{minipage}%
    \begin{minipage}{.26\textwidth}
        \centering
        \includegraphics[height=2cm]{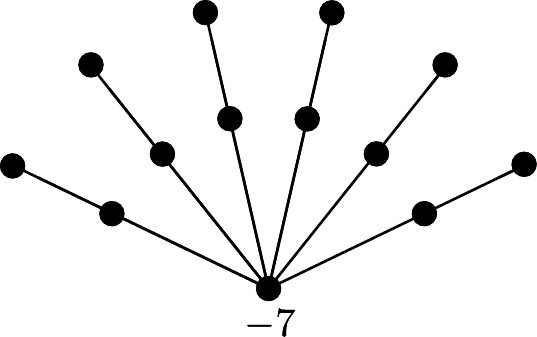}
        \caption{}
        \label{fig:7bdgraph6}
    \end{minipage}%
\end{figure}

\section{On the completeness of our lists of relations}
\label{sec:completeness}
		
We will now discuss the minimal and maximal Euler characteristic for the filling of a plumbing corresponding to a product of twists of the form $T_{b_1}^{a_1}\cdots T_{b_{n-1}}^{a_{n-1}}T_{b_n}$ which is equivalent to a product of non-boundary parallel twists.

Because the Euler characteristic increases as the number of Dehn twists increases, we can consider the relations that involve the maximum and minimum number of boundary-parallel twists.

In particular, by Theorem~\ref{thm:generalization2}, for a product of twists $T_{b_1}^{a_1}\cdots T_{b_{n-1}}^{a_{n-1}}T_{b_n}$, if $1+\sum_{i=1}^{n-1}a_i\leq 2n-5$, then there is no relation involving this product and some product of non boundary parallel twists. Equivalently, if there are $2n-5$ or fewer twists, no such relation exists. Thus, the minimum number of Dehn twists must be $2n-4$ or greater.

Additionally, taking $i=n-2$ in Theorem~\ref{thm:relationgeneral} gives us a relation between 

$$[T_{b_1}\cdots T_{b_{n-2}}][T_{b_{n-1}}^{n-3}]T_{b_n}$$

and a product of twists over non boundary parallel curves.

We note that the number of Dehn twists in this product is $n-2+n-3+1=2n-5+1=2n-4$. Since a relation involving this number of Dehn twists exists, we can say that this is the minimum number of boundary-parallel Dehn twists needed for a nontrivial relation of the specified form.

The Euler characteristic of the symplectic plumbing given by this product of twists is $2-n+2n-4=n-2$ and the plumbing graph is as shown in Figure~\ref{fig:7bdgraph7}. This is the minimal possible Euler characteristic of a plumbing for which there exists a different symplectic filling given by a relation of the form we are considering.

Furthermore, the maximal Euler characteristic for such a plumbing can be found by taking the highest number of boundary-parallel Dehn twists represented by any relation of the specified form. 

Consider a relation 

$$T_{b_1}^{a_1}\cdots T_{b_{n-1}}^{a_{n-1}}T_{b_n}=T_{\alpha_1}\cdots T_{\alpha_m}$$ where each $\alpha_i$ is some non boundary parallel curve. Since the left hand side includes exactly one twist over the outer boundary, the braid corresponding to both sides of the relation must have a linking number of $-1$ for all pairs of boundary components. By Lemma~\ref{lemma:linkingnumber}, no two boundary components can both be contained in multiple $\alpha_i$'s. Therefore, the maximum multiplicity of any boundary component $b_i$ is $n-2$. This is because we can twist over each curve containing exactly $b_i$ and one other interior boundary component. This results in a collection of $\binom{n-1}{2}$ non boundary parallel curves on the right hand side.

In fact, using Theorem~\ref{thm:relationgeneral} and choosing $i=2$, we have
$$[T_{b_1}^{n-3}T_{b_2}^{n-3}][T_{b_3}^{n-3}\cdots T_{b_{n-1}}^{n-3}]T_{b_n}=T_{b_1,b_2}[T_{b_3,b_2}T_{b_3,b_1}]\cdots[T_{b_{n-1},b_{n-2}}\cdots T_{b_{n-1},b_1}]$$

Note that the choice of curves on the right hand side is exactly the collection of $\binom{n-1}{2}$ curves described above. Thus, there exists such a relation that gives the maximum possible multiplicity to each interior boundary component. 

On the left hand side, the total number of Dehn twists is $(n-3)(n-1)+1$, so the maximum Euler characteristic of a plumbing given by a relation of the specified form is $2-n+(n-3)(n-1)+1=n^2-5n+6$. 

\begin{prop}
		Up to isometry, the relations described in Proposition~\ref{prop:5bdrelations} are the only relations on the mapping class monoid of $S_{0,0}^5$ which are of the form
		
		$$T_{b_1}^{a_1}\cdots T_{b_{n-1}}^{a_{n-1}}T_{b_n}=T_{\alpha_1}\cdots T_{\alpha_m}$$
	
	where each $a_i$ is a positive integer and each $\alpha_i$ is a convex, non boundary parallel simple closed curve.
	\end{prop}

\begin{proof} Assume that $T_{b_1}^{a_1}\cdots T_{b_{n-1}}^{a_{n-1}}T_{b_n}=T_{\alpha_1}\cdots T_{\alpha_m}$ is a relation on the mapping class monoid of $S_{0,0}^5$. We know that the braid corresponding to the left hand side must be as in Figure \ref{fig:braid51}.

		\begin{figure}
		\centering

		\includegraphics[height=2cm]{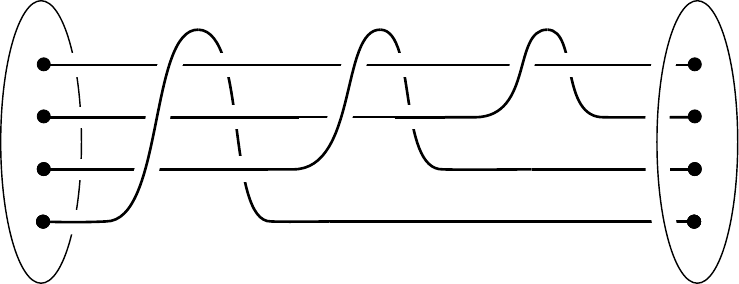} 
\caption{}
\label{fig:braid51}
	\end{figure}
	
	In order for the relation to hold, we know this braid must also correspond to the right hand side. Thus, we see that the pairwise linking numbers must all be equal to $-1$. By Lemma~\ref{lemma:linkingnumber}, this means that, for any two boundary components there must be exactly one curve $\alpha_i$ that contains both boundary components in its interior. Furthermore, because the $\alpha_i$'s are not boundary parallel, each must contain at least 2 interior boundary components and no more than 3 interior boundary components. 
	
	In the case where there is a convex curve $\alpha_i$ that contains exactly three interior boundary components, up to isometry, there is only one such arrangement for a curve. Such a curve must contain three adjacent boundary components. Without loss of generality, assume they are $b_1,b_2$, and $b_3$. Any other curves included in the collection of $\alpha_i$'s may not contain more than one of these boundary components. Furthermore, we still need a linking number of $-1$ between $b_4$ and each of $b_1,b_2$, and $b_3$. Therefore, we must also twist over each curve that contains $b_4$ and exactly one of the other interior boundary components. We may not include any additional curves because we now have a linking number of $-1$ for all pairs of interior boundary components. We note that Relation 1 in Proposition~\ref{prop:5bdrelations} describes this case.
	
	In the case where none of the $\alpha_i$'s contain more than two boundary components, we must include the $\binom{4}{2}$ curves that each contain exactly two interior boundary components in the twists on the right hand side. This describes the case covered by Relation 2 in Proposition~\ref{prop:5bdrelations}.
	
	Since these are the only two cases, we conclude that Proposition~\ref{prop:5bdrelations} presents all relations of the specified form, up to isometry.
\end{proof}
	
	We present a similar proposition for $S_{0,0}^6$:
	
\begin{prop}
Up to isometry, Proposition~\ref{prop:6bdrelations} presents all relations on the mapping class monoid of $S_{0,0}^6$ of the form

$$T_{b_1}^{a_1}\cdots T_{b_{n-1}}^{a_{n-1}}T_{b_n}=T_{\alpha_1}\cdots T_{\alpha_m},$$

where each $a_i$ is a positive integer and each $\alpha_i$ is a convex, non boundary parallel curve on $S_{0,0}^6$.
\end{prop}

\begin{figure}
\begin{minipage}{.3\textwidth}
\centering
\includegraphics[height=3cm]{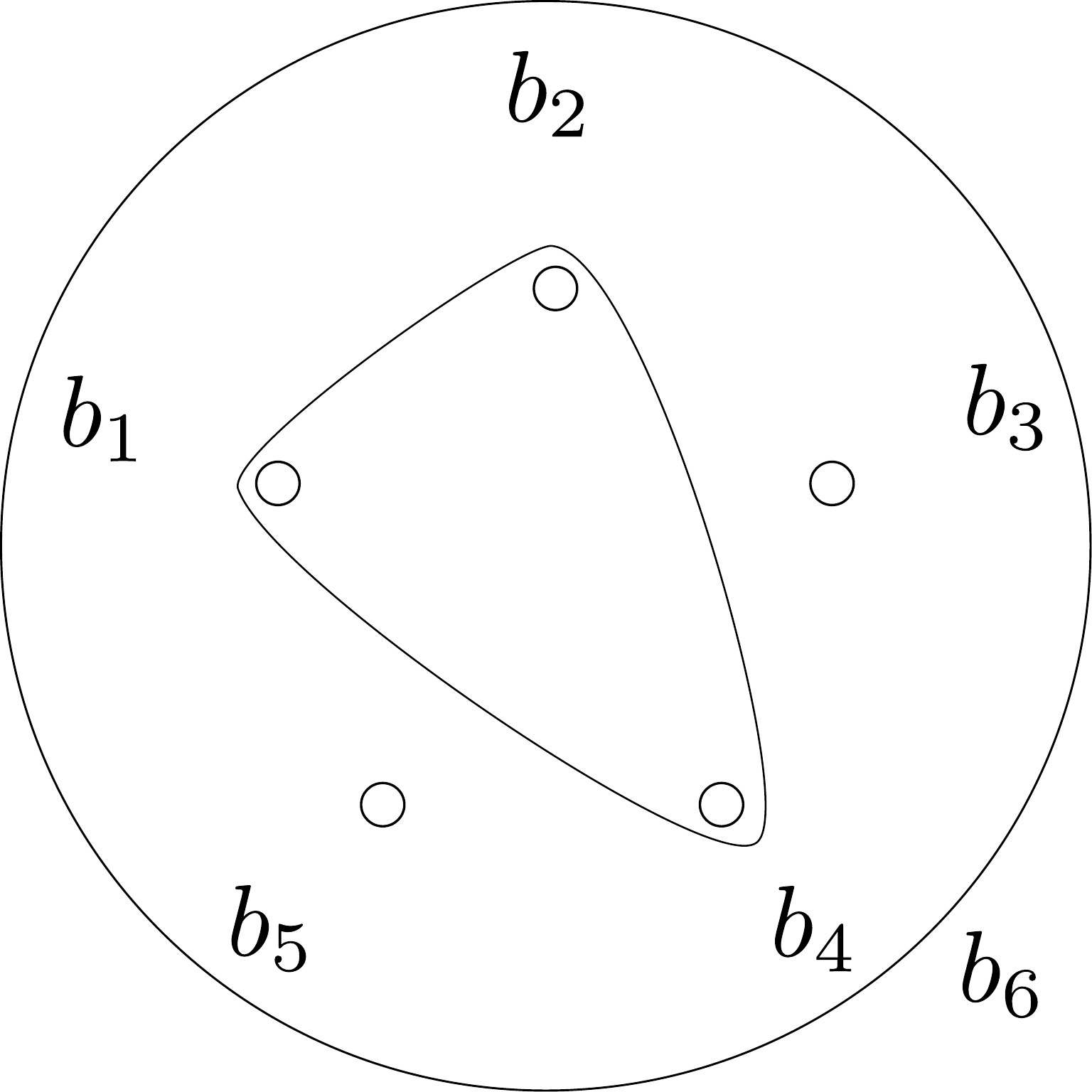}
\caption{}
\label{fig:firstcurve}
\end{minipage}
\begin{minipage}{.3\textwidth}
\centering
\includegraphics[height=3cm]{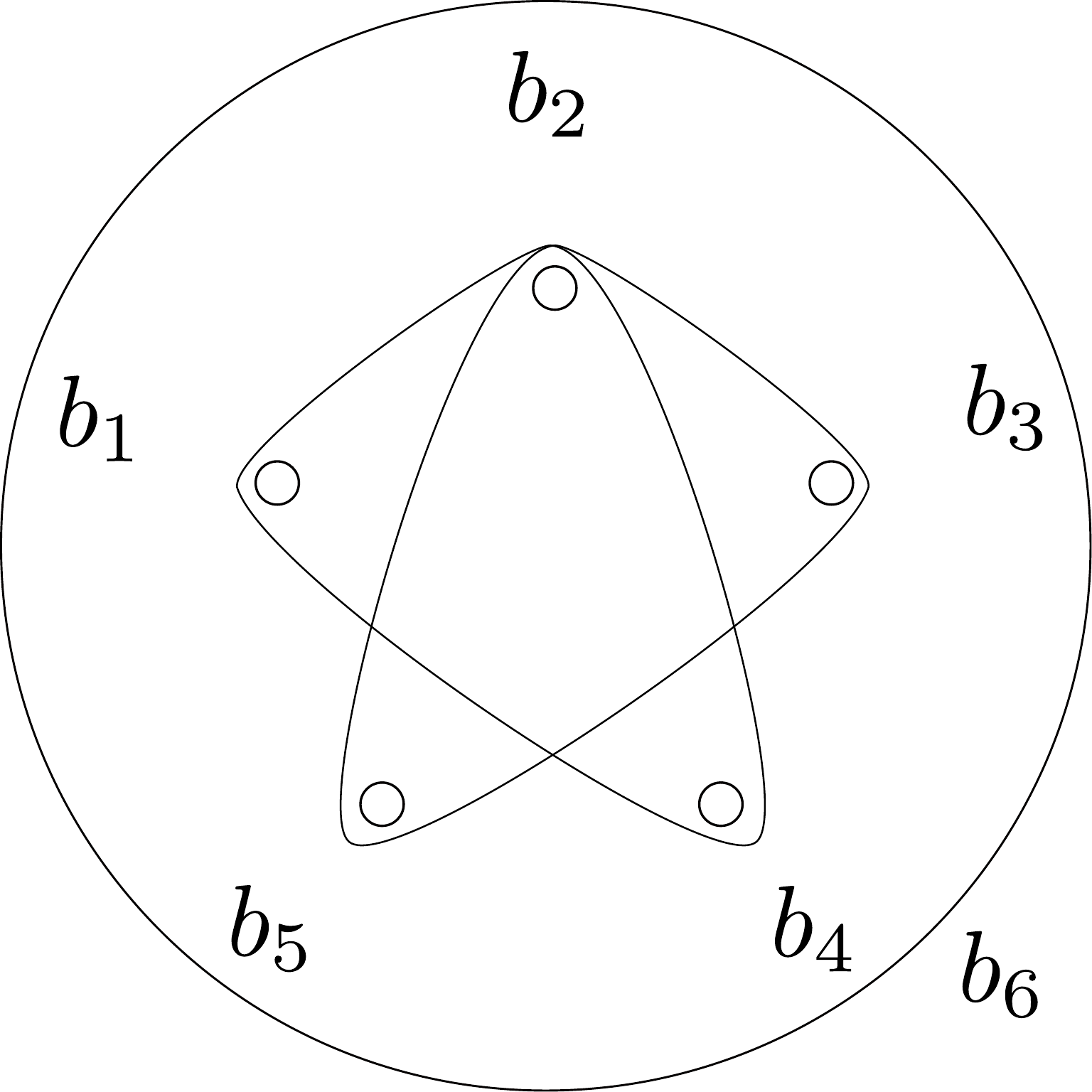}
\caption{}
\label{fig:secondcurve} 
\end{minipage}
\begin{minipage}{.3\textwidth}
\centering
\includegraphics[height=3cm]{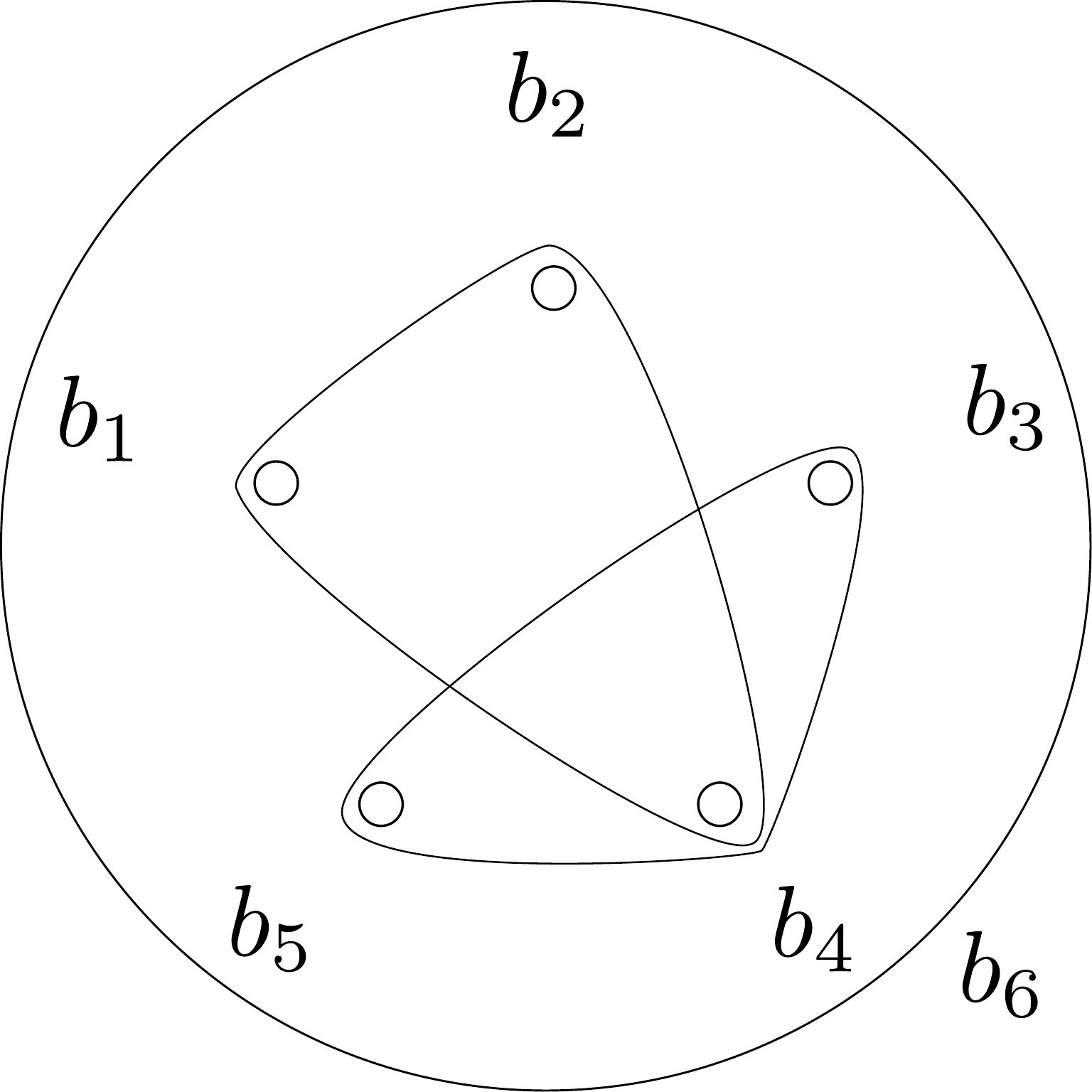} 
\caption{}
\label{fig:thirdcurve}
\end{minipage}
\end{figure}

\begin{figure}
\begin{minipage}{.3\textwidth}
\centering
\includegraphics[height=3cm]{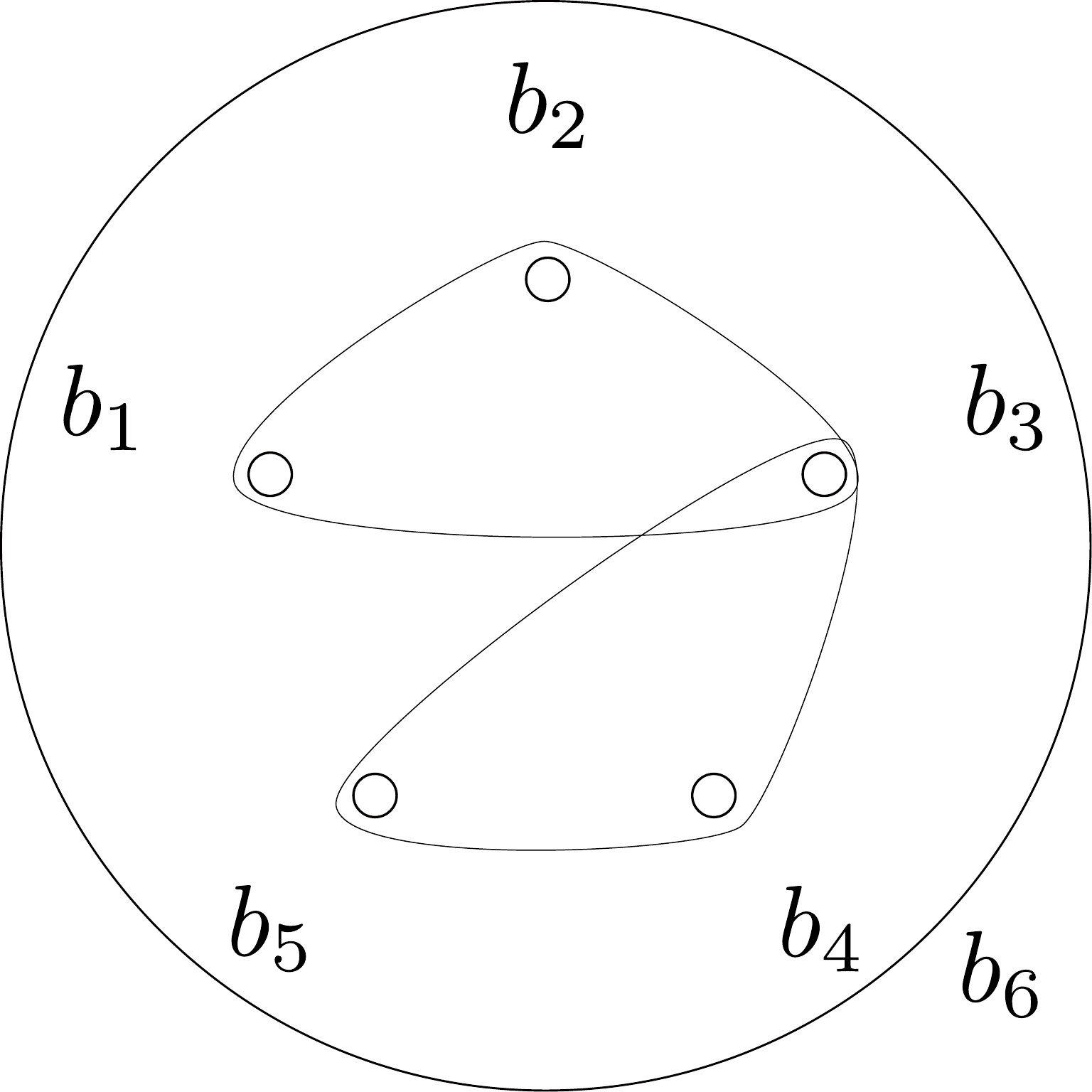}
\caption{}
\label{fig:fourthcurve}
\end{minipage}
\end{figure}

\begin{proof} Using the same reasoning as above, we will consider all possible collections of non boundary parallel curves $\alpha_i$ on $S_{0,0}^6$ such that for any two boundary components, there is exactly one curve that contains both in its interior.

Since a curve that contains all 5 interior boundary components is parallel to the outer boundary, any $\alpha_i$ must contain at most 4 interior boundary components. 

In the case where there is one curve that contains four boundary components, we can assume without loss of generality that $b_5$ is not contained in this curve. Then we see that we must additionally include all $\binom{4}{2}$ curves that contain $b_5$ and exactly one other interior boundary component. Thus, up to isometry, Relation 3 in Proposition~\ref{prop:6bdrelations} describes this case.

Now we will consider the case where at least one curve $\alpha_i$ contains exactly three interior boundary components. Up to isometry, there are two possibilities: the boundary components are all adjacent or they are not. 

In the case that they are not adjacent, we have the curve shown in Figure \ref{fig:firstcurve}.

There are three possible cases for the remaining curves: there is another curve with three non adjacent boundary components, there is a curve with three adjacent boundary components, and there are no other curves with three boundary components. 

In the case that all other curves $\alpha_i$ contain two boundary components, we get a collection of curves as represented by Relation 7.

In the case that there is another curve with three non adjacent boundary components, up to isometry, we will have the curves shown in Figure \ref{fig:secondcurve}. Then all other curves $\alpha_i$ must contain exactly two boundary components each. This case is represented by Relation 4.

In the case that there is a curve containing three adjacent boundary components, up to isotopy, we have the curves shown in Figure \ref{fig:thirdcurve}. All remaining curves must contain exactly two boundary components, describing the collection of curves used in Relation 5.

Now we can consider the case where there is a curve containing exactly three adjacent boundary components and no curves containing three non adjacent boundary components. 

If this is the only curve with more than two boundary components, all other curves must contain two boundary components each. Up to isotopy, this is described by the right-hand side of Relation 2. 

If there is another curve with three adjacent boundary components, then up to isometry we have an arrangement as shown in Figure \ref{fig:fourthcurve}. All remaining curves must contain exactly two interior boundary components each. Up to isometry, this collection of curves is described by Relation 6.

Finally, if all curves $\alpha_i$ have no more than two boundary components, we must have all $\binom{6}{2}$ curves that contain exactly two boundary components each. This case is described by Relation 1. 

Since this covers all possible cases, we see that, up to isometry, Proposition~\ref{prop:6bdrelations} describes all relations of the specified form.
\end{proof}

We suspect that Proposition~\ref{prop:7bdrelations} describes all possible relations of the specified form on the mapping class monoid of $S_{0,0}^7$. However, since they do not include all possible collections of curves $\alpha_i$ which would give the correct pairwise linking numbers, in order to prove this claim, it must be shown that some of the possible collections would not give a relation with any product of boundary parallel twists.

We have a weaker claim that we can prove for $S_{0,0}^7$:

\begin{prop}
The plumbing graphs in Corollary~\ref{cor:7bdplumbinggraphs} describe all possible plumbings that correspond to a relation on the mapping class monoid of $S_{0,0}^7$ of the form 

$$T_{b_1}^{a_1}\cdots T_{b_6}^{a_6}T_{b_7}=T_{\alpha_1}\cdots T_{\alpha_m},$$

where each $a_i$ is a positive integer and each $\alpha_i$ is a non boundary parallel curve.
\end{prop}

\begin{proof}
We will show this by considering a relation $T_{b_1}^{a_1}\cdots T_{b_{n-1}}^{a_{n-1}}T_{b_n}=T_{\alpha_1}\cdots T_{\alpha_m}$ on the mapping class monoid of $S_{0,0}^7$. In order for this relation to hold, we know that, for any pair of interior boundary components, there must be exactly one curve $\alpha_i$ that contains both in its interior. Furthermore, since each $\alpha_i$ is not boundary parallel, we know that each one must contain at least 2 and at most 5 boundary components in its interior. 

Consider the case where there is an $\alpha_i$ that contains exactly 5 boundary components in its interior. Let $b_j$ denote the interior boundary component not contained in $\alpha_i$. This means that we must include all curves that contain $b_j$ and exactly one other boundary component and no additional curves. Therefore, under the right hand side, the multiplicity of $b_j$ is 5 and the multiplicity of all other interior boundary components is 2. Because the multiplicity of each boundary component must be equal under both sides of the relation, the left hand side must be $T_{b_1}\cdots T_{b_i}^4\cdots T_{b_6}T_{b_7}$. The graph for the corresponding plumbing shown in Figure~\ref{fig:7bdgraph7}. 

Note that the plumbing graph is equivalent under any homeomorphism that rearranges the boundary components, so this hold regardless of the relative position of the interior boundary components. This is true for all plumbing graphs, so we can consider all subsequent cases up to relabeling by any permutation.

In the case that there is one curve $\alpha_i$ that contains exactly 4 boundary components, we can either have the case where all other $\alpha_i$'s contain exactly two boundary components or the cause where there is a curve $\alpha_j$ which contains exactly three boundary components. In the former case, we see that the boundary components inside of $\alpha_j$ will each have a multiplicity of 3 and those outside will have a multiplicity of 5. Up to relabeling, this means that the left hand side must be $T_{b_1}^2T_{b_2}^2T_{b_3}^2T_{b_4}^2T_{b_5}^4T_{b_6}^4T_{b_7}$. Therefore, any relation of this form will correspond to the plumbing graph shown in Figure~\ref{fig:7bdgraph4}. In the latter case, we know that $\alpha_i$ and $\alpha_j$ must contain exactly one boundary component in common and all pairs of boundary components not both in $\alpha_i$ or $\alpha_j$ must be contained in a curve with exactly those two boundary components. Up to relabelling, a product of twists over such a collection of curves would be equivalent to $T_{b_1}^2T_{b_2}^2T_{b_3}^2T_{b_4}T_{b_5}^3T_{b_6}^3T_{b_7}$, which corresponds to the plumbing graph in Figure~\ref{fig:7bdgraph2}.

In the case where all curves contain at most three boundary components in their interior, we have three cases: there is exactly one such curve, there are two such curves that have one boundary component in common, and there are two such curves that do not have any boundary components in common. In the first case, all other curves must contain exactly two boundary components, and there must be one that contains any pair of boundary components that are not both in the first curve. Up to relabelling, the product of twists over these curves will be equivalent to $T_{b_1}^3T_{b_2}^3T_{b_3}^3T_{b_4}^4T_{b_5}^4T_{b_6}^4T_{b_7}$, which corresponds to the plumbing graph in Figure~\ref{fig:7bdgraph3}. In the case where there are two curves with three boundary components that have one boundary component in common, all remaining curves must contain exactly two boundary components, and this will be equivalent to $T_{b_1}^3T_{b_2}^2T_{b_3}^3T_{b_4}^3T_{b_5}^3T_{b_6}^4T_{b_7}$ up to relabelling. This corresponds to the graph in Figure~\ref{fig:7bdgraph5}. Finally, if the two curves do not have any boundary components in common, then each additional curve must contain exactly one boundary component from each of these curves. Up to relabelling, this is equivalent to $T_{b_1}^3T_{b_2}^3T_{b_3}^3T_{b_4}^3T_{b_5}^3T_{b_6}^3T_{b_7}$, which corresponds to the plumbing graph given in Figure~\ref{fig:7bdgraph6}.

Finally, in the case where all curves contain at most two boundary components, we must include twists over each of the $\binom{6}{2}$ curves that contain exactly two interior boundary components. Any product of twists over this collection of curves that corresponds to the specified braid is equivalent to $T_{b_1}^4T_{b_2}^4T_{b_3}^4T_{b_4}^4T_{b_5}^4T_{b_6}^4T_{b_7}$ and corresponds to the graph in Figure~\ref{fig:7bdgraph1}. 

Since we have considered all possible collections of curves $\alpha_i$, we see that the graphs in Corollary~\ref{cor:7bdplumbinggraphs} represent all possible plumbings that are built from relations of the specified form on the mapping class monoid of $S_{0,0}^7$.
\end{proof}

We can also consider the possible relations of the form $T_{b_1}^{a_1}\cdots T_{b_{n-1}}^{a_{n-1}}T_{b_n}=T_{\alpha_1}\cdots T_{\alpha_m}$, this time allowing the $\alpha_i$'s to be non convex. Because Lemma \ref{lemma:linkingnumber} applies for convex and non convex curves, using the same combinatorial argument as in the proof of Proposition \ref{prop:7bdrelations}, we can conclude that the plumbing graph and associated Euler characteristics for any relation of this form must be the same as those associated with a relation involving only convex curves. Therefore, we must find another way to determine whether the resulting fillings are distinct from those given in the existing relations.

\begin{question}
Do there exist relations of the form $T_{b_1}^{a_1}\cdots T_{b_{n-1}}^{a_{n-1}}T_{b_n}=T_{\alpha_1}\cdots T_{\alpha_m}$, where at least one $\alpha_i$ is non convex, that cannot be related by a global diffeomorphism to a relation involving only convex curves?
\end{question}


\newpage

\begin{thebibliography}{9}

\bibitem[AO01]{AO}
Selman Akbulut and Burak Ozbagci.
\newblock Lefschetz fibrations on compact {S}tein surfaces.
\newblock {\em Geom. Topol.}, 5:319--334, 2001.

\bibitem[EG08]{EndoGurtas}
Hisaaki Endo and Yusuf Z. Gurtas. Lantern relations and rational blowdowns. \textit{Proc. Amer. Math. Soc.} 138:1131-1142, 2010.
	
	\bibitem[EMVHM11]{EMVHM}
	Hisaaki Endo, Thomas~E. Mark, and Jeremy Van Horn-Morris.
	\newblock Monodromy substitutions and rational blowdowns.
	\newblock {\em J. Topol.}, 4(1):227--253, 2011.
	
	\bibitem[FM12]{FarbMargalit} Benson Farb and Dan Margalit.
	\newblock{A Primer on Mapping Class Groups}.
	\newblock{\em Princeton University Press}, 2011.
	
\bibitem[Ful03]{Fuller}
Terry Fuller.
   \newblock{Lefschetz fibrations of 4-dimensional manifolds}.
   \newblock{\em Cubo Mat. Educ.},
   5(3): 275--294, 2003.
	
	\bibitem[GM13]{GayMark}
	David Gay and Thomas~E. Mark.
	\newblock Convex plumbings and {L}efschetz fibrations.
	\newblock {\em J. Symplectic Geom.}, 11(3):363--375, 2013.
	
	\bibitem[Gir02]{Giroux}
Emmanuel Giroux.
\newblock G\'{e}om\'{e}trie de contact: de la dimension trois vers les
   dimensions sup\'{e}rieures.
\newblock In {\em Proceedings of the {I}nternational {C}ongress of
   {M}athematicians, {V}ol. {II} ({B}eijing, 2002)}, pages 405--414.
Higher Ed.
   Press, Beijing, 2002.
   
   \bibitem[GS99]{GS}
Robert~E. Gompf and Andr\'{a}s~I. Stipsicz.
\newblock {\em {$4$}-manifolds and {K}irby calculus}, volume~20 of {\em
   Graduate Studies in Mathematics}.
\newblock American Mathematical Society, Providence, RI, 1999.

\bibitem[KS16]{KarakurtStarkston}
\c{C}a\u{g}ri Karakurt and Laura Starkston. 
\newblock Surgery along star-shaped plumbings and exotic smooth structures on 4-manifolds. 
\newblock \textit{Algebr. Geom. Topol.} 16(3):1585-1635, 2016

\bibitem[LP01]{LP}
Andrea Loi and Riccardo Piergallini.
\newblock Compact {S}tein surfaces with boundary as branched covers of
{$B^4$}.
\newblock {\em Invent. Math.}, 143(2):325--348, 2001.

\bibitem[MM09]{MargalitMcCammond}
Dan Margalit and Jon McCammond.
\newblock{Geometric presentations for the pure braid group.}
\newblock {\em J. Knot Theory Ramifications}, 18(1): 1--20, 2009.
	
	\bibitem[McD90]{McDuff}
	Dusa McDuff.
	\newblock The structure of rational and ruled symplectic {$4$}-manifolds.
	\newblock {\em J. Amer. Math. Soc.}, 3(3):679--712, 1990.
	
	\bibitem[OS04]{OzbagciStipsicz} Burak Ozbagci and Andr\'{a}s I. Stipsicz. Surgery on Contact 3-Manifolds and Stein Surfaces. \emph{Springer-Verlag Berlin Heidelberg}, 2004.
	
	\bibitem[PS21]{PlamenevskayaStarkston}
	Olga Plamenevskaya and Laura Starkston.
	\newblock Unexpected Stein fillings, rational surface singularities, and plane curve arrangements. preprint, arXiv:2006.06631 [math.GT].
	
	
	\bibitem[PVHM10]{PVHM}
	Olga Plamenevskaya and Jeremy Van Horn-Morris.
	\newblock Planar open books, monodromy factorizations and symplectic fillings.
	\newblock {\em Geom. Topol.}, 14(4):2077--2101, 2010.
	
	\bibitem[Sym98]{Symington} Margaret Symington. Symplectic rational blowdowns. \textit{J. Differential Geom.}, 50:505-518, 1998 
	
	\bibitem[TW75]{TW}
W.~P. Thurston and H.~E. Winkelnkemper.
\newblock On the existence of contact forms.
\newblock {\em Proc. Amer. Math. Soc.}, 52:345--347, 1975.
	
	\bibitem[Wen10]{Wendl}
	Chris Wendl.
	\newblock Strongly fillable contact manifolds and {$J$}-holomorphic foliations.
	\newblock {\em Duke Math. J.}, 151(3):337--384, 2010.
	
\end{thebibliography}
\end{document}